\input ./style/arxiv-general.cfg
\documentclass[aop,MSNbibl,seceqn,dvips]{arximspdf}
\makeatletter
   \@ifpackageloaded{graphicx}{}{\usepackage{graphicx}}
\makeatother
\usepackage{mathbh}

\doi{10.1214/15-AOP1011}
\volume{44}
\issue{3}
\pubyear{2016}
\firstpage{1817}
\lastpage{1863}
\docsubty{FLA}

\makeatletter
\def\mathds{\mathbh}
\newcommand{\eqref}[1]{(\ref{#1})}

\newcommand{\inftwo}[2]{\mathop{\inf_{#1}}_{#2}} 
\renewcommand{\b}{\beta}

\newcommand{\1}{\mathbh{1}}
\newcommand{\var}{\operatorname{Var}}

\newcommand{\tmix}{T_{\mathrm{ mix}}}
\newcommand{\trel}{T_{\mathrm{ rel}}}

\newcommand{\D}{\Delta}
\renewcommand{\b}{\beta}
\renewcommand{\L}{\Lambda}

\renewcommand{\a}{\alpha}
\renewcommand{\t}{\tau}

\newcommand{\z}{\zeta}
\newcommand{\e}{\varepsilon}
\renewcommand{\r}{\rho}

\renewcommand{\O}{\Omega}

\newcommand{\gap}{\mathrm{gap}}

\newtheorem{theorem}{Theorem}[section]

\newtheorem{lemma}[theorem]{Lemma}
\newtheorem{proposition}[theorem]{Proposition}

\newproclaim{remark}[theorem]{Remark}
\newproclaim{remarks}[theorem]{Remarks}

\newtheorem{claim}[theorem]{Claim}
\newproclaim{definition}[theorem]{Definition}

\newtheorem{maintheorem}{Theorem}

\newtheorem{theorema}{Theorem}[section]
\newtheorem{lemmaa}[theorema]{Lemma}
\newproclaim{remarka}[theorema]{Remark}
\newtheorem{claima}[theorema]{Claim}

\newcommand{\N}{\mathbb N}
\newcommand{\cA}{\mathcal A}
\newcommand{\cB}{\mathcal B}
\newcommand{\cC}{\mathcal C}
\newcommand{\cD}{\mathcal D}
\newcommand{\cE}{\mathcal E}
\newcommand{\cK}{\mathcal K}
\newcommand{\cL}{\mathcal L}
\newcommand{\cS}{\mathcal S}
\newcommand{\cT}{\mathcal T}
\newcommand{\cZ}{\mathcal Z}

\newcommand{\bbE}{{\mathbb E}}
\newcommand{\bbN}{{\mathbb N}}
\newcommand{\bbP}{{\mathbb P}}
\newcommand{\bbR}{{\mathbb R}}

\def\a{\alpha}
\def\b{\beta}
\def\e{\varepsilon}
\def\h{\eta}
\def\p{\pi}
\def\r{\rho}
\def\s{\sigma}
\def\t{\tau}
\def\z{\zeta}
\def\D{\Delta}
\def\G{\Gamma}
\def\L{\Lambda}
\def\O{\Omega}

\renewcommand{\div}{\operatorname{div}}

\makeatother

\begin{document}
\begin{frontmatter}

\title{Relaxation to equilibrium of generalized East processes on
${\mathbb Z}
^{d}$:
Renormalization group analysis and energy-entropy competition}
\runtitle{Relaxation to equilibrium of generalized East processes}

\begin{aug}
\author[A]{\fnms{Paul}~\snm{Chleboun}\ead[label=e1]{p.i.chleboun@warwick.ac.uk}\thanksref{T1}},
\author[B]{\fnms{Alessandra}~\snm{Faggionato}\ead[label=e2]{faggiona@mat.uniroma1.it}}
\and
\author[C]{\fnms{Fabio}~\snm{Martinelli}\corref{}\ead[label=e3]{martin@mat.uniroma3.it}}
\affiliation{Warwick University, Universit\`a La Sapienza and
Universit\`a Roma Tre}
\runauthor{P. Chleboun, A. Faggionato and F. Martinelli}
\thankstext{T1}{Supported
by the University of Warwick IAS through a Global Research Fellowship.}
\address[A]{P. Chleboun\\
Mathematics Institute and\\
Centre for Complexity Science\\
Warwick University\\
Coventry CV4 7AL\\
United Kingdom\\
\printead{e1}}
\address[B]{A. Faggionato\\
Dipartimento di Matematica\\
Universit\`a La Sapienza\\
P.le Aldo Moro 2\\
00185 Roma\\
Italy\\
\printead{e2}}
\address[C]{F. Martinelli\\
Dipartimento di Matematica e Fisica\\
Universit\`a Roma Tre \\
Lg. San Murialdo 1\\
00146 Roma\\
Italy\\
\printead{e3}}
\end{aug}

%
\received{\smonth{4} \syear{2014}}
%
\revised{\smonth{1} \syear{2015}}

%
\begin{abstract}
We consider a class of kinetically constrained interacting particle
systems on ${\mathbb Z}^d$ which
play a key role in several heuristic qualitative and quantitative
approaches to describe
the complex behavior of glassy
dynamics. With
rate one and independently among the vertices of ${\mathbb Z}^d$, to each
occupation variable $\eta_x\in\{0,1\}$ a new value is proposed by
tossing a $(1-q)$-coin. If a certain local constraint is satisfied by
the current
configuration the proposed move is accepted, otherwise it is
rejected. For $d=1$, the constraint requires that there is a vacancy at
the vertex to the
left of the updating vertex. In this case, the process is the
well-known East process. On ${\mathbb Z}^2$,
the West or the South neighbor of the updating vertex must contain a
vacancy, similarly, in higher dimensions. Despite of their apparent
simplicity, in the limit $q\searrow0$ of low vacancy density,
corresponding to a low temperature
physical setting, these processes feature a rather complicated dynamic
behavior with hierarchical relaxation time scales,
heterogeneity and universality. Using renormalization group ideas, we
first show
that the relaxation time on ${\mathbb Z}^d$ scales as the $1/d$-root of the
relaxation time of the East process, confirming indications coming from
massive numerical simulations.
Next, we compute the relaxation time in finite boxes by carefully
analyzing the subtle energy-entropy
competition, using a multiscale
analysis, capacity methods and an algorithmic construction. Our results
establish dynamic heterogeneity and a
dramatic dependence on the boundary conditions. Finally, we prove a
rather strong anisotropy property of these processes: the creation of
a new vacancy at a vertex $x$ out of an isolated one at the origin (a
seed) may occur on
(logarithmically) different time scales which heavily depend not only
on the $\ell_1$-norm of $x$ but also on its direction.
\end{abstract}

%
\begin{keyword}[class=AMS]
\kwd{60K35}
\kwd{82C20}
\end{keyword}

\begin{keyword}
\kwd{East model}
\kwd{relaxation time}
\kwd{spectral gap}
\kwd{glassy dynamics}
\kwd{renormalization group}
\end{keyword}
%
\end{frontmatter}

\section{Introduction}\label{sec1}
The East process is a one-dimensional spin system introduced in the
physics literature by J\"{a}ckle and Eisinger~\cite{JACKLE} in 1991
to model the behavior of cooled liquids near the glass transition
point, specializing a class of models that goes back to~\cite{FH}.
Each site $x\in{\mathbb Z}$ carries a $\{0,1\}$-value
(vacant/occupied) denoted
by $\eta_x$. The process attempts to update $\eta_x$ to $1$ at rate
$0<p<1$ (a parameter)
and to $0$ at rate $q=1-p$, only accepting the proposed update if
$\eta_{x-1}=0$ (a ``kinetic constraint''). Since the constraint at
site $x$ does not depend on the spin at $x$, it is straightforward to verify
that the product $\operatorname{Bernoulli}(1-q)$ measure is a reversible measure.

Despite of its apparent simplicity, the East model
has attracted much attention both in the physical and in the
mathematical community (see, e.g.,
\cite{SE1,SE2,CDG,Aldous,East-Rassegna}). It in fact
features a
surprisingly rich behavior, particularly when $q\ll1$ which
corresponds to a low temperature setting in the physical
interpretation, with a host of phenomena
like mixing time cutoff and front propagation
\cite{East-cutoff,Blondel}, hierarchical coalescence and universality
\cite{FMRT-cmp} and dynamical heterogeneity
\cite{CFM,CFM-JSTAT}, one of the main signatures
of glassy dynamics. Dynamical heterogeneity is strongly associated to a
\emph{broad spectrum}
of relaxation time scales which emerges as the result of a
subtle \emph{energy-entropy} competition. Isolated vacancies
with, for example, a block of $N$ particles to their left, cannot in fact
update unless the system injects enough
additional vacancies in a cooperative way in order to
unblock the target one. Finding the correct time scale on which this
unblocking process occurs requires a highly nontrivial analysis to
correctly measure the energy contribution (how many extra vacancies are
needed) and the entropic one (in how many ways the
unblocking process may occur). The final outcome is a very nontrivial
dependence of the corresponding characteristic time scale on the
equilibrium vacancy density $q$ and on the block
length $N$ (cf.
\cite{CFM}, Theorems 2 and 5).

Mathematically, the East model poses
very challenging and interesting problems because of the hardness of
the constraint and the fact that it is not attractive. It also has
interesting ramifications in combinatorics \cite{CDG},
coalescence processes \cite{FMRT,FRT,FMRT-cmp} and random walks on
triangular matrices \cite{Peres-Sly}. Moreover, some of the
mathematical tools
developed for the analysis of its relaxation time scales proved to
be quite powerful also in other contexts such as card shuffling problems
\cite{Bhatnagar:2007tr} and random evolution of surfaces
\cite{PietroCaputo:2012vl}. Finally, it is worth mentioning that some
attractive conjectures which
appeared in the physical literature on the basis of numerical
simulations, had to be thoroughly revised
after a sharp mathematical analysis \cite{CMRT,CFM,CFM-JSTAT}.

Motivated by a series of nonrigorous contributions on \emph{realistic}
models of glass formers (cf. \cite{Garrahan,Garrahan2003,Keys2013}),
in this paper we examine for the first time a natural generalization of
the East process
to the higher dimensional lattice ${\mathbb Z}^d$, $d>1$, in the sequel
referred to as the \emph{East-like}
process. In one dimension, the East-like process coincides with the
East process. In $d=2$, the process evolves
similarly to the East process but now the kinetic
constraint requires that the South \emph{or} West neighbor of the
updating vertex contains \emph{at least} one vacancy analogously in
higher dimensions.

An easy comparison argument with the one-dimensional case shows that the
East-like process is always ergodic, with a relaxation time
$\trel({\mathbb Z}^d;q)$ which is bounded from above by
$\trel({\mathbb Z};q)$.\setcounter{footnote}{1}\footnote{Notice that the more constrained North--East
model in which the constraint at $x$
requires that \emph{both} the West \emph{and} the South neighbor of
$x$ contains a vacancy
has a ergodicity breaking transition when $p$ crosses the oriented
critical percolation value.} However, massive numerical simulations
\cite{Garrahan} suggest
that $\trel({\mathbb Z}^d;q)$ is much smaller than $\trel({\mathbb
Z};q)$ and
that, as
$q\searrow0$, it scales as
$\trel({\mathbb Z};q)^{1/d}$, where the $1/d$-root is a signature of several
different effects on the cooperative
dynamics of the sparse vacancies: the entropy associated with the
number of ``oriented'' paths
over which a vacancy typically sends a wave of influence and the
energetic cost of creating the required number of vacancies.

Our first result (cf. Theorem~\ref{th:main1} below) confirms the above conjecture by a novel combination
of renormalization group ideas and block dynamics on one hand and an
algorithmically built bottleneck using capacity methods on the other.

Our second result analyzes the relaxation time in a finite box. In
this case, in order to guarantee the irreducibility of the chain, some
boundary conditions must be introduced by declaring
\emph{unconstrained} the spins belonging to certain subsets of the
boundary of $\L$. For example, in two dimensions one could imagine to
freeze to the
value $0$ all the spins belonging to the South--West (external) boundary
of the
box. In this case, we say that we have \emph{maximal} boundary
conditions. If instead all the spins belonging to the South--West
(external) boundary are frozen to be~$1$ with the exception of one spin
adjacent to the South--West corner then we say that we have
\emph{minimal} boundary conditions. In Theorem~\ref{th:main2},
we compute the precise asymptotic as $q\searrow0$ of
the relaxation
time with maximal and minimal boundary conditions and show that there
is a dramatic difference between the two. The result extends also to
mixing times.

The third result concerns another
time scale which is genuinely associated with the out-of-equilibrium
behavior. For simplicity, consider the process
on ${\mathbb Z}^2$ and, starting from the configuration with a single vacancy
at the origin, let $T(x;q)$ be the mean hitting time of the set $\{\eta
\dvtx
\eta_x=0\}$ where $x$ is some vertex in the first quadrant. In other
words, it is the mean time that it takes for
the initial vacancy at the origin to create a vacancy at $x$. Here, the
main outcome is a strong dependence of $T(x;q)$ as $q\searrow0$ not
only on
the $\ell_1$-norm $\|x\|_1$ but also on the \emph{direction} of $x$
(cf. Theorem~\ref{th:main3}
and Figure~\ref{fig:1bis}). When $\log_2 \|x\|_1\gg\sqrt{\log_2
1/q,}$ the
process proceeds much faster (on a logarithmic scale) along the
diagonal direction than along the coordinate axes. If instead $\|x\|
_1=O(1)$ as
$q\searrow0$, then the asymptotic behavior of $T(x;q)$ is essentially
dictated by
$\|x\|_1$. This crossover phenomenon is yet another
instance of the key role played by the energy-entropy competition in
low temperature kinetically constrained models.

Finally, in the \hyperref[app]{Appendix} we have collected some results on the
exponential rate of decay of the \emph{persistence function} $F(t)$,
that is, the
probability for the stationary infinite volume East-like process that
the spin at the origin does not flip before time $t$. Such a rate of
decay is often used by physicists as a proxy for the inverse relaxation
time. For the East model, we indeed prove that the latter assumption is
correct. In higher dimension, we show that the above rate of decay
coincides with that of the time auto-correlation of the spin at the
origin. Our results are quite similar to those obtained years ago
for the Ising model by different methods \cite{Holley}.

We point out that in \cite{EPL} we have provided an overview of the
results and mathematical tools of this paper, with special emphasis to
the connections with the existing physics literature on the subject.

\subsection{Outline of the paper}
In the next section, we define the model and quantities of interest,
in Section~\ref{sec:main} we state our main results.
In Section~\ref{troppo}, we collect various technical tools:
monotonicity, graphical construction, block dynamics, capacity methods
and the bottleneck inequality. Section~\ref{bottiglia} is devoted to an
algorithmic construction of an efficient bottleneck and it will
represent the key ingredient for the proof of the various lower bounds in
Theorems \ref{th:main2} and \ref{th:main3}. Theorems \ref{th:main1},
\ref{th:main2} and~\ref{th:main3} are proved in Sections~\ref{sec:5},
\ref{sec:6} and \ref{sec:7}, respectively. Although these proofs have
been divided into different sections, they are actually linked. In
particular, the proof of the upper bound in Theorem~\ref{th:main1}
uses the upper bound for $n \leq\theta_q$ in \eqref{eq:3} of
Theorem~\ref{th:main2} and the proof of the upper bound in \eqref{eq:2} for
$n \geq\theta_q /d$ of Theorem~\ref{th:main2} uses the upper bound in Theorem~\ref
{th:main1}.
Finally, we have collected in the \hyperref[app]{Appendix} some results on the exponential
rate of decay of the persistence function.

\section{Model and main results}
\subsection{Setting and notation}
\label{setting} Given the $d$-dimensional lattice ${\mathbb Z}^d$, we let
${\mathbb Z}_+^d:=\{x=(x_1,\ldots,x_d)\in{\mathbb Z}^d \dvtx x_i\ge1
\ \forall i\le
d\}$.
Given $x\in{\mathbb Z}^d$ and $A\subset
{\mathbb Z}^d$, we let $\|x\|_1:=\sum_{i=1}^d |x_i|$ and $\|A\|_1:=
\sup_{x,y
\in A} \|x-y\|_1$. A box in ${\mathbb Z}^d$ will be any set $\L$ of
the form
$\prod_{i=1}^d
[a_i,b_i], a_i\leq b_i\ \forall i$, where here and in the
sequel it is understood that the interval $[a_i,b_i]$ consists of all
the points $x\in{\mathbb Z}$ with $a_i\le x\le b_i$. We call the vertices
$(a_1,\ldots,a_d)$ and $(b_1,\ldots,b_d)$ the lower and upper corner of
$\L$, respectively.

Let $\cB:= \{ e_1, e_2, \ldots, e_d\}$ be the canonical basis of
${\mathbb Z}^d$. The \emph{East-like} boundary of a box $\L$, in the sequel
$\partial_E
\L$, is the set
\[
\partial_E \L:= \bigl\{x \in{\mathbb Z}^d\setminus\L
\dvtx x+e \in\L \mbox{ for some } e \in\cB \bigr\}.
\]
Given $\D\subset{\mathbb Z}^d$, we will denote by $\O_\D$ the
product space $ \{0,1\}^\D$ endowed with the product topology. If $\D=
{\mathbb Z}^d$, we simply write $\O$. In the sequel, we will refer to the
vertices of $\D$ where a given configuration $\eta\in\O_\D$ is
equal to
one (zero) as the \emph{particles} (\emph{vacancies}) of $\eta$. Given
two disjoint sets $V, W \subset
{\mathbb Z}^d$ together with $(\xi,\eta)\in\O_V \times\O_W$, we
denote by
$\xi\eta$ the
configuration in $\O_{V \cup W}$ which coincides with $\xi$ in $V$ and
with $\eta$ in $W$. If $V\subset\D$ and $\eta\in\O_\D$, we will write
$\eta_V$ for the restriction of $\eta$ to $V$.

For any box $\L$, a configuration $\s\in
\O_{\partial_E \L}$ will be referred to as a \emph{boundary condition}.
A special role is assigned to the following class of
boundary conditions.
%
\begin{definition}
\label{cond_ergodico}
Given a box $\L= \prod_{i=1}^d [a_i,b_i] $, we will say that a boundary condition $\s$ is
\emph{ergodic} if there exists $e \in\cB$ such that
$\s_{a-e}=0$, $a= (a_1, \ldots, a_d)$.
We call the boundary condition identically equal to
zero \emph{maximal}. If instead $\s$ is such that by
removing one vacancy in $\s$ one obtains a nonergodic boundary
condition then $\s$ is said to be \emph{minimal}. Equivalently, $\s$ is
minimal if it has a unique vacancy at $a-e$ for some $e\in\cB$. Notice
that for
$d=1$ the maximal and minimal boundary conditions coincide.
\end{definition}
%
\subsection{The finite volume East-like process}
Given a box $\L$
and an ergodic boundary configuration $\s$, we
define the
\emph{constraint} at site $x\in\L$ with boundary
condition $\s$ as the indicator function on $\O_\L$
\[
c_x^{\L, \s} (\eta):= \mathds{1}_{\{\omega\dvtx  \exists e \in\cB\
\mathrm{such\ that\ } \omega_{x-e}=0\}} (\eta\s).
\]
Then the East-like process with parameter $q \in(0,1)$ and boundary
configuration $\s$ is the continuous time Markov chain with state space
$\O_\L$ and infinitesimal generator
%
\begin{eqnarray}
\label{eq:gen} \cL^\s_\L f(\eta) &=& \sum
_{x \in\L} c_x^{\L, \s}(\eta) \bigl[
\eta_x q+ (1-\eta_x)p \bigr] \cdot \bigl[ f \bigl(
\eta^x \bigr) -f(\eta) \bigr]
\nonumber
\\[-8pt]
\\[-8pt]
\nonumber
&=& \sum_{x \in\L} c_x^{\L, \s}(
\eta) \bigl[\pi_x (f)- f \bigr] (\eta),
\end{eqnarray}
where $p:= 1-q$, $\eta^x$ is the configuration in $\O_\L$ obtained
from $\eta$ by flipping its value at $x$ and $\pi_x$ is the
$\operatorname{Bernoulli}(p)$ measure on the spin at $x$.

Since the local constraint $c_x^{\L, \s}(\eta)$ does not depend on
$\eta_x$ and the boundary condition is ergodic, it is simple to check
that the East-like process is an ergodic chain reversible
w.r.t. the product $\operatorname{Bernoulli}(p)$ measure $\pi_\L=\prod_{x\in\L
}\pi_x$
on $\O_\L$.
We will denote by $\bbP^{\L,\s}_\eta(\cdot)$ and
$\bbE^{\L,\s}_\eta(\cdot)$ the law and the associated expectation of
the process started from $\eta$.

\begin{remark}
When $d=1$, the East-like process coincides with the well-known East
process.
\end{remark}
Next, we recall the definition of spectral gap and relaxation time. To
this aim, given $f\dvtx \O_\L\to\bbR$ and $V \subset\L$, we define
$\var_V(f)$ as the conditional variance of $f$ w.r.t. to $\pi_V$ given
the variables outside $V$. The quadratic
form or \emph{Dirichlet} form associated to $-\cL^\s_\L$ will be
denoted by $\cD^\s_\L$ and it takes the form
%
\begin{equation}
\label{eq:DirForm} \cD^\s_\L(f):=\pi_\L \bigl(
f \bigl(- \cL^\s_\L f \bigr) \bigr) =\sum
_{x\in\L
}\pi_\L \bigl(c^{\L,\s}_x
\var_x(f) \bigr).
\end{equation}
%
\begin{definition}[(Relaxation time)]
The smallest positive eigenvalue of $-\cL^\s_\L$
is called the spectral gap and it is denoted by
$\gap(\cL_{\L}^\s)$. It satisfies the Rayleigh--Ritz variational principle
%
\begin{equation}
\label{eq:gap} \gap \bigl( \cL^\s_\L \bigr):= \inftwo{f
\dvtx \O_\L\mapsto\bbR} {f\ \mathrm{nonconstant} } \frac{ \cD^\s_\L(f) }{\var_\L(f) }.
\end{equation}
The relaxation time $T^\s_\mathrm{ rel} (\L)$ is defined as the inverse of
the spectral gap:
%
\begin{equation}
\label{rilasso} T_\mathrm{ rel}^\s(\L)= \frac{1}{ \gap(\cL^\s_\L)}.
\end{equation}
Equivalently, the relaxation time is the best constant $c$ in the
Poincar\'e inequality
\[
\var_\L(f)\le c \cD^\s_\L(f)\qquad \forall f.
\]
\end{definition}

\subsection{The infinite volume East-like process}
We now define the East process on the entire lattice ${\mathbb Z}^d$. Let
$
c_x (\eta):= \mathds{1}_{\{\omega\dvtx  \exists e \in\cB\ \mathrm{such\ that\ }
\omega_{x-e}=0\}}(\eta)$,
be the constraint at $x$.
Then the East-like process on ${\mathbb Z}^d$ is the continuous time Markov
process with state space $\O$, with reversible measure given by the
product $\operatorname{Bernoulli}(p)$ measure $\pi=\prod_{x\in{\mathbb Z}^d}\pi_x$ and
infinitesimal generator $\cL$ whose
action on functions depending on finitely many spins is given by
%
\begin{eqnarray}\label{eq:generator}
\cL f(\eta) &=& \sum_{x \in{\mathbb Z}^d} c_x(\eta)
\bigl[ \eta_x q+ (1-\eta _x)p \bigr] \cdot \bigl[ f \bigl(
\eta^x \bigr) -f(\eta) \bigr]
\nonumber
\\[-8pt]
\\[-8pt]
\nonumber
&=&\sum_{x \in{\mathbb Z}^d} c_x(\eta)
\bigl[\pi _x (f)- f \bigr] (\eta).
\end{eqnarray}
We will denote by $\bbP_\eta(\cdot)$ and
$\bbE_\eta(\cdot)$ the law and the associated expectation of
the process started from $\eta$.
We will also denote by $\gap(\cL)$ and $T_\mathrm{ rel} ({\mathbb Z}^d)$
the spectral
gap and relaxation time defined similar to the finite
volume case.

It is a priori not obvious that $T_\mathrm{ rel}
({\mathbb Z}^d)<+\infty$ for all values of $q\in(0,1)$. However, we
observe that
the East-like process is \emph{less} constrained than a infinite
collection of independent
one-dimensional East processes, one for every line in ${\mathbb Z}$ parallel
to one of the coordinate axis, each of which has a finite relaxation
time \cite{Aldous}; hence the conclusion.\vspace*{1pt} A formal proof goes
as follows. Define $c_x^\mathrm{ East}(\eta)=\mathds{1} (\eta
_{x-e_1}=0 )$ and observe that $c_x(\eta)\ge c_x^\mathrm{ East}(\eta
)$. Therefore, the
Dirichlet form $\cD(f)=\sum_x \pi (c_x \var_x(f) )$ of the
East-like process is bounded from below by $\sum_x \pi (c^\mathrm{
East}_x \var_x(f) )$
which is nothing but the Dirichlet form of a collection of independent
East processes, one for every line in ${\mathbb Z}^d$
parallel to the first coordinate axis. The Rayleigh--Ritz variational
principle for the spectral gap implies that $T_\mathrm{ rel}
({\mathbb Z}^d)$ is not larger that the relaxation time of the above product
process. In turn, by the tensorization property
of the spectral gap (see, e.g., \cite{Saloff}), the relaxation
time of the product process coincides with that of the one-dimensional
East process $\trel({\mathbb Z})$.
In conclusion,
%
\begin{equation}
\label{eq:1} T_\mathrm{ rel} \bigl({\mathbb Z}^d \bigr)\le
T_\mathrm{ rel}({\mathbb Z})\qquad \forall d\ge1.
\end{equation}

\subsection{Main results}
\label{sec:main}
In order to present our main results, it will be convenient to fix some extra
notation. First, since
we will be interested in the small $q$ regime, the dependence
on $q$ of the various time scale characterizing the relaxation
toward equilibrium will be added to their notation.
Second, the finite volume East-like process with maximal or minimal
boundary conditions will exhibit
quite different relaxation times for $d\ge2$ and, therefore, they will
have a
special notation. More precisely:
\begin{itemize}
\item if the boundary condition $\s$ outside a box $\L$ is
\emph{maximal \textup{(}minimal\textup{)}} we
will write $T_\mathrm{
rel}^\mathrm{ max}(\L;q)$ [$T_\mathrm{
rel}^\mathrm{ min}(\L;q)$] instead of $T_\mathrm{
rel}^{\s}(\L;q)$.
\item In the special case in which $\L$ is the cube $[1,L]^d$ of
side $L$, we will write
$T_\mathrm{ rel}^\s(L;q)$ instead of $T_\mathrm{ rel}^\s(\L)$.
\end{itemize}
With the above notation, the first theorem pins down the dependence on
the dimension $d$ of the relaxation time for the process on ${\mathbb Z}^d$.
Before stating it, we recall the precise asymptotic of
$\trel({\mathbb Z};q) $ as $q\downarrow0$. Let
$\theta_q:=\log_2(1/q)$. In \cite{Blondel2}, Lemma~6.3, it was proved
that, for any $L\ge2^{\theta_q}$,
\[
\trel(L;q)=2^{O(\theta_q)}\trel \bigl(2^{\theta_q};q \bigr),
\]
with $O(\theta_q)$ uniform in $L$. In turn, the relaxation time on scale
$2^{\theta_q}$ is given by (cf. \cite{CFM}, Theorem~2)
$2^{{\theta_q^2}/{2} + \theta_q\log_2 \theta_q +O(\theta_q)}$. By combining the above estimates (cf. also Lemma~\ref{fin-vol}), we
conclude that
%
\begin{equation}
\label{eq:East} \trel({\mathbb Z};q)=2^{{\theta_q^2}/{2} + \theta_q\log_2
\theta_q
+O(\theta_q)}.
\end{equation}
%
\begin{maintheorem}
\label{th:main1}
As $q\downarrow0$
\[
\trel \bigl({\mathbb Z}^d;q \bigr) =2^{ ({\theta^2_q}/{(2d)})(1+o(1))}.
\]
In particular
\[
\trel \bigl({\mathbb Z}^d;q \bigr)= \trel({\mathbb
Z};q)^{({1}/d) (1+o(1))}.
\]
\end{maintheorem}
%
\begin{remark}\label{controllo_errori}
The above divergence of the relaxation time as $q\downarrow0$
confirms the indications coming from numerical simulation
(\cite{Garrahan}, Figure~3 in Section~9). Our proof will also show that the $o(1)$
correction is $\O (\frac{1}{\theta_q}\log_2\theta_q )$ and
$O(\theta_q^{-1/2})$.\footnote{Recall that
$f=O(g)$, $f=o(1)$ and $f=\O(g)$ mean that $|f|\le C |g|$ for
some constant $C$, $f\to0$ and $\limsup|f|/|g| >0$, respectively.}
\end{remark}
The second result analyzes the relaxation time in a finite box. The
main outcome here is a dramatic dependence on the boundary conditions
in dimension greater than one.
%
\begin{maintheorem}
\label{th:main2}
1. Let $\L=[1,L]^d$ with $L\in(2^{n-1},2^n]$ and $n=n(q)$ such
that $\lim_{q\downarrow0}n(q)=+\infty$. Then, as $q\downarrow0$,
%
\begin{eqnarray}
\label{eq:2} \trel^\mathrm{ max}(L;q)&=& \cases{ 2^{ (n\theta_q -
 d{n\choose2} )(1+o(1))}, &\quad $
\mbox{for $n\le \theta_q/d$}$,\vspace*{2pt}
\cr
2^{ ({\theta^2_q}/{(2d)})(1+o(1))},& \quad$
\mbox{otherwise}$, }
\\
\label{eq:3}\trel^\mathrm{ min}(L;q)&=& \cases{ 2^{n\theta_q - {n\choose2}
 +n\log_2 n +O(\theta_q)},&\quad$\mbox{for $n \le\theta_q$}$,\vspace*{2pt}
\cr
2^{ {\theta^2_q}/{2}+\theta_q\log_2 \theta_q +O(\theta
_q)},&\quad $\mbox {otherwise,}$}
\end{eqnarray}
where the constant entering in $O(\theta_q)$ in \eqref{eq:3} does not
depend on the choice of $n=n(q)$.\vspace*{-3pt}
\begin{longlist}[2.]
\item[2.] Fix $n\in\bbN$ and let $\L=[1,L]^d$ with $\|\L\|_1 +1 \in
(2^{n-1},2^n]$. Then, as $q\downarrow0$,
%
\begin{equation}
\label{eq:333} \trel^\mathrm{ min}(L;q)= 2^{n\theta_q +O_n(1)},
\end{equation}
where $O_n(1)$ means that the
constant may depend on $n$.
\end{longlist}
\end{maintheorem}
%
\begin{remark}
Notice that $L_c=2^{\theta_q/d}$ is the characteristic intervacancy
distance at equilibrium (the average number of
vacancies in a box of side $L_c$ is one). It coincides with the
characteristic length above which the relaxation time with maximal
boundary conditions starts to scale with $q$ like the infinite volume
relaxation time.

With minimal boundary conditions the
relaxation time behaves as in the one-dimensional case
(\cite{CFM}, Theorem~2). In particular, the critical scale $2^{\theta
_q}=1/q$ is the
equilibrium inter-vacancy distance in $d=1$. For what concerns
\eqref{eq:333}, we observe that $\|\L\|_1 +1$ is the number of vertices
in any (East-like) oriented path\footnote{That is, a path in
the oriented graph $\vec{\mathbb Z}^d$ obtained by orienting each edge
of the
graph ${\mathbb Z}^d$ in the direction of increasing
coordinate-value.}connecting $x_*=(1,\ldots,1)$ to
$v^*=(L,\ldots,L)$. With this interpretation the leading term in the
RHS of \eqref{eq:333} coincides with the leading term of the
relaxation time for an East process on such an oriented path
(cf., e.g., \cite{East-Rassegna}).

Finally, let $\tmix^{\s}(L;q)$ be the \emph{mixing time} of the
East-like process
with boundary conditions $\s$, that is, the smallest time $t$ such that,
for all starting configurations, the law at time $t$ has total
variation distance from $\pi_\L$ at most $1/4$ (cf., e.g., \cite{Levin2008}).
It is well known (see, e.g., \cite{Saloff}) that
\[
\trel^\s(L;q)\le\tmix^\s(L;q)\le\trel^\s(L;q)
\bigl(1-\tfrac
{1}{2} \log \pi^* \bigr),
\]
where $\pi^*:=\min_\eta\pi_\L(\eta)=q^{|\L|}$. Thus $\tmix^\mathrm{
max}(L;q)$ and $\tmix^\mathrm{ min}(L;q)$ satisfy
the first bound in \eqref{eq:2} and \eqref{eq:3}, respectively.
\end{remark}

%
\begin{remark}
The error term $o(1)$ in \eqref{eq:2} can be somewhat detailed [cf.
Remark~\ref{patria} and estimate \eqref{tricolore} in Section~\ref{sec:6}].
\end{remark}
%
In order to state the last result, we need to introduce a new time
scale. For any $x\in{\mathbb Z}_+^d$, let $\t_x$ be the hitting time
of the set
$\{\eta \dvtx  \eta_{x}=0\}$ for the East-like process in
${\mathbb Z}^d_+$ with some ergodic boundary condition $\s$ and let
$T^\s(x;q):=\bbE^{\s}_{\mathds1}(\t_x)$ be its mean when the starting
configuration has no vacancies (here and in the sequel denoted by
$\mathds1$).
For
simplicity, we present our result on the asymptotics of $T^\s(x;q)$ as
$q\downarrow0$ only for minimal boundary conditions [e.g.,
corresponding to a single vacancy at $(1,0,\ldots,0)$] since they
correspond to
the most interesting setting from the physical point of
view. In this case the mean hitting times $T^\mathrm{ min}(x;q)$ give some
insight on how a wave of vacancies originating from a single one
spreads in space--time.
Other boundary conditions could be treated as well. Moreover, we
restrict ourselves only to two main directions for the vertex $x$:
either the diagonal (i.e., $45^\circ$ degrees in $d=2$) or along
one of the
coordinate axes.
%
\begin{maintheorem}
\label{th:main3}
1. Let $v_*=(L,1,\ldots,1), v^*=(L,L,\ldots,L)$ with $L\in
(2^{n-1},\break 2^n]$ and $n=n(q)$ with $\lim_{q\downarrow0}n(q)=+\infty$.
Then, as $q\downarrow0$,
%
\begin{equation}
\label{hitmin} T^\mathrm{ min}(v_*;q)= 2^{n\theta_q - {n\choose2} +n\log_2
n +O(\theta_q)} \qquad\mbox{for } n
\le \theta_q, 
\end{equation}
whereas for the vertex $v^*$ the mean
hitting time satisfies
%
\begin{equation}
\label{hitmax} T^\mathrm{ min} \bigl(v^*;q \bigr)= 2^{n\theta_q - d{n\choose2} +O(\theta_q\log
\theta_q)},
\end{equation}
for all $ n \le
\theta_q/d $.
\begin{longlist}[2.]
\item[2.] Fix $n\in\bbN$ and let $x\in{\mathbb Z}_+^d$ be such that
$\|x-x_*\|_1+1\in[2^{n-1},2^n)$ where $x_*=(1,\ldots,1)$. Then, as
$q\downarrow0$,
%
\begin{equation}
\label{hitmin2} T^\mathrm{ min}(x;q)= 2^{n\theta_q +O_n(1)}.
\end{equation}
\end{longlist}
\end{maintheorem}
%
\begin{remark}
\label{pasqua}
Actually, we shall prove that \eqref{hitmax} holds for any
ergodic boundary conditions on $\partial_E {\mathbb Z}_+^d$ and not
just for
the minimal ones.
\end{remark}
%

\begin{figure}[b]
\centering
\begin{tabular}{@{}cc@{}}

\includegraphics{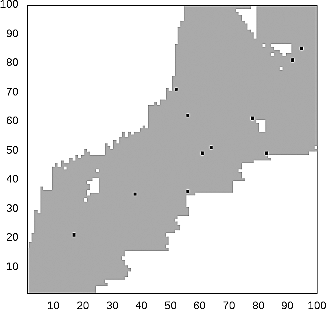}
 & \includegraphics{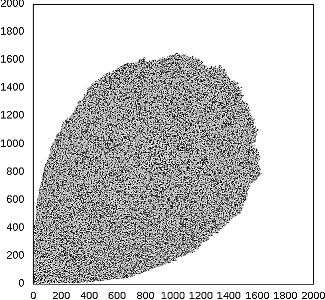}\\
\footnotesize{(a)} & \footnotesize{(b)}
\end{tabular}
\caption{A snapshot of a simulation of the East-like process with
minimal boundary conditions
and initial condition constantly identically equal to $1$. White dots
are vertices that have never been updated, grey dots correspond to
vertices that have been updated at
least once and the black dots are the vacancies present in the
snapshot. \textup{(a)} $q=0.002, t = 3 \times10 ^ {12}$; \textup{(b)} $q=0.25, t=9 \times10 ^ {3}$.}
\label{fig:1bis}
\end{figure}

The above result highlights a somewhat unexpected directional behavior
of the
East-like process (cf. Figure~\ref{fig:1bis}). Take for simplicity
minimal boundary conditions, $d=2$ and
$n=\theta_q/2$ so that $L=2^n$ is the mean intervacancy distance $L_c$
at equilibrium. Despite of the fact that the $\ell_1$ distance from
the origin of $v^*$ is roughly twice that of $v_*$,
implying that the process has to create more vacancies out of $\mathds
1$ in order to reach $v^*$ compared to those needed to reach $v_*$,
the mean hitting time for $v_*$ is much
larger (as $q\downarrow0$ and on a logarithmic scale) than the mean
hitting time for $v^*$. The main reason for such a surprising
behavior is the fact that $v^*$ is connected to the single vacancy of
the boundary condition by
an exponentially large (in $\|x\|_1$) number of (East-like) oriented
paths while $v_*$ is connected by only one such path. When, for
example, $n\propto\theta_q$ this entropic effects can
compensate the increase in energy caused by the need to use more
vacancies. The phenomenon could disappear for values of
$n=O(\sqrt{\theta_q\log\theta_q})$ for which the term $n\choose2$
becomes comparable
to the error term $O(\theta_q\log\theta_q)$. It certainly does so for
$n=O(1)$ as shown
in \eqref{hitmin2}.
%
\begin{remark}
One may wonder what is the behavior of the mean hitting time $T^\mathrm{ min}(x;q)$
when $q$ is fixed and $\|x\|_1\to\infty$. If $x$ belongs to, for
example, the
half-line $\{x\in{\mathbb Z}^d_+ \dvtx x_i=1\ \forall i\ge2\}$ and
since the
projection on this line of the East-like process with minimal boundary
conditions
is the standard East process, one can conclude
(cf. \cite{Blondel,East-cutoff}) that $\lim_{\|x\|_1\to\infty
}T^\mathrm{
min}(x;q)/\|x\|_1$ exists. Simulations suggest [cf. Figure~\ref{fig:1bis}(b)] that the same occurs for points $x$ belonging to
suitable rays through
the origin
but that in this case the limit is smaller than the one obtained along
the coordinate axes. Moreover, it seems natural to conjecture that the random
set $\cS_t$ consisting of all points of ${\mathbb Z}_+^d$ that have been
updated at least once before time $t$,
after rescaling by $t$ satisfies a shape theorem.
\end{remark}


\section{Some preliminary tools}
\label{troppo}
In this section, we collect some technical tools to guarantee a smoother
flow of the proof of the main results.
\subsection{Monotonicity} It is clear from the variational
characterization of the spectral gap that any monotonicity of the Dirichlet
form of the East-like (e.g., in the boundary conditions, in the
volume or in the constraints) induces a similar monotonicity of the
spectral gap and, therefore, of the
relaxation time. In what follows, we collect few simple useful
inequalities.
%
\begin{lemma}\label{lem:ineq}
Let $\L=\prod_{i=1}^d[a_i,b_i]$ and let
$\L'=\prod_{i=1}^d[a_i,b'_i]$, with $b'_i\ge b_i\ \forall i$. Fix two
ergodic boundary conditions $\s,\s'$ for $\L,\L'$, respectively, such
that $\s_x\le\s'_x$ for all $x\in\partial_E\L$. Then
%
\begin{equation}
\label{eq:degrado} \trel^\s(\L;q)\le\trel^{\s'} \bigl(
\L';q \bigr).
\end{equation}
In particular, $\trel^\mathrm{ max}(L;q)$ and $\trel^\mathrm{ min}(L;q)$ are
nondecreasing function
of $L$. Moreover,
%
\begin{eqnarray}
\label{eq:myrta2}\trel^\mathrm{ max}(\L;q)&\le&\trel^{\s}(\L;q)\le
\trel ^\mathrm{ min}(\L;q),
\\
\label{eq:myrta} \trel^\mathrm{ max}(\L;q)&\le&\trel \bigl({\mathbb
Z}^d;q \bigr).
\end{eqnarray}
\end{lemma}
\begin{pf}
The inequality $c_x^{\L,\s'}\le c_x^{\L,\s}$ implies that
$\cD_\L^{\s'}(f)\le\cD_\L^{\s}(f)$. Moreover, for any function
$f\dvtx \O
_{\L}\mapsto\bbR$, it
holds that $\var_\L(f)=\var_{\L'}(f)$ and $\cD^{\s}_{\L}(f)=
\cD^{\s}_{\L'}(f)$. The first two statements \eqref{eq:degrado} and
\eqref{eq:myrta2} are immediate
consequences of the
variational characterization of the spectral gap. The last statement
follows by similar arguments (cf. \cite{CMRT}, Lemma~2.11).
\end{pf}
The second result establishes a useful link between the finite volume
relaxation time with maximal boundary conditions and the infinite
volume relaxation time.
%
\begin{lemma}
\label{fin-vol}
$\trel({\mathbb Z}^d;q)= \lim_{L\to
\infty}\trel^\mathrm{ max}(L;q)$.
\end{lemma}
\begin{pf}
Using \eqref{eq:myrta} together with the fact that $\trel^\mathrm{
max}(L;q)$ is nondecreasing in $L$, it is enough to show that
\[
\trel^\mathrm{ max} \bigl({\mathbb Z}^d;q \bigr)\le\sup
_L \trel^\mathrm{ max}(L;q).
\]
That indeed follows from \cite{CMRT}, proof of Proposition~2.13.
\end{pf}
%
%
\subsection{Graphical construction}
\label{sec:graph}
It is easily seen that the East-like process (in finite or infinite
volume) has the following graphical
representation (see, e.g., \cite{CMRT}). To
each $x\in{\mathbb Z}^d$, we associate a rate one Poisson process and,
independently, a family of independent $\operatorname{Bernoulli}(p)$ random variables
$\{s_{x,k} \dvtx k \in\N\}$. The occurrences of the Poisson process
associated to $x$ will be denoted by $\{t_{x,k} \dvtx k \in\N\}$. We
assume independence as $x$ varies in ${\mathbb Z}^d$. This fixes the
probability space whose probability law will be denoted by
$\bbP(\cdot)$. Expectation w.r.t. $\bbP(\cdot)$ will be denoted by
$\bbE
(\cdot)$. Notice that, $\bbP$-almost surely, all
the occurrences $\{t_{x,k} \dvtx k \in\N, x\in{\mathbb Z}^d\}$ are different.
On the above probability space we construct a Markov process according
to the following rules. At each time
$t_{x,k}$, the site $x$ queries the state of its own constraint $c_x$
(or $c_x^{\L,\s}$ in the finite volume case).
If and only if the constraint is satisfied ($c_x = 1$ or $c_x^{\L,\s
}=1$), then $t_{x,k}$
is called a \emph{legal ring} and the configuration resets its
value at site $x$ to the value of the corresponding Bernoulli variable
$s_{x,k}$. A simple consequence of the graphical construction is that
the projection on a finite box
$\L$ of the form $\L=\prod_{i=1}^d[1,L_i]$ of the East-like process on
${\mathbb Z}^d_+$ with boundary
condition $\s$ coincides with the
East-like process on $\L$ with boundary conditions given by the
restriction of $\s$ to $\partial_E\L$.
\subsection{A block dynamics version of the East-like process}
\label{block process}
Let $S$ be a finite set and let $\mu$ be a probability measure on
$S$. Let $G\subset S$ and define $q^*=1-p^*=\mu(G)$. Without loss of
generality, we assume that $q^*\in
(0,1)$. On
$\O^*=S^{{\mathbb Z}^d}$ consider the Markov process with generator
$\cA$ whose action on functions depending on finitely many
coordinates is given by [cf. \eqref{eq:generator}]
%
\begin{equation}
\label{eq:3bis} \cA f(\omega)=\sum_{x\in{\mathbb Z}^d}c^*_x(
\omega) \bigl[\mu _x(f)-f \bigr](\omega),
\end{equation}
where $\mu_x(f)(\omega)= \sum_{\omega_x\in S}\mu(\omega
_x)f(\omega)$ is the conditional
average on the coordinate $\omega_x$ given $\{\omega_y\}_{y\neq x}$ and
$c^*_x(\omega)$
is the indicator of the event that, for some $e\in\cB$, the
coordinate $\omega_{x-e}$ belongs to the subset $G$.
%
\begin{remark}
Exactly as for the East-like process there is a finite volume version
of the above process on a box $\L$ with an ergodic boundary condition
$\s\in
S^{\partial_E\L}$ and generator $\cA_\L^\s$. In particular, $\s$ is
maximal if $\s_x\in G$ for all $x\in\partial_E\L$, and in this case we
will write $\cA^\mathrm{ max}_\L$.
\end{remark}
If
$S=\{0,1\}$, $G=\{0\}$ and $\mu$ is the $\operatorname{Bernoulli}(p)$
measure on $S$,
the above process coincides with the East-like process. As for the
latter, one easily verifies reversibility
w.r.t. the product measure with marginals at each site $x$ given by
$\mu$. The above process also admits a
graphical construction tailored for the applications we have in
mind.

Similar to the East-like
process one associates to each $x\in{\mathbb Z}^d$ a rate one Poisson
process, a family of independent $\operatorname{Bernoulli}(p^*)$ random variables
$\{s_{x,k} \dvtx k \in\N\}$ and a family of independent random variables
$\{\omega_{x,k} \dvtx k \in\N\}\in S^{\bbN}$, such that $\omega_{x,k}$
has law
$\mu(\cdot\mid G^c)$ if $s_{x,k}=1$ and $\mu(\cdot\mid G)$
otherwise. All the above variables are independent as $x$ varies in
${\mathbb Z}^d$.
One then constructs a Markov process according to the following rules.
At each time
$t_{x,k}$, the site $x$ queries the state of its own constraint $c^*_x$.
If and only if the constraint is satisfied ($c^*_x = 1$), then the
configuration resets its
value at site $x$ to the value of the corresponding variable
$\omega_{x,k}$. The law of the process started from $\omega$ will be denoted
by $\bbP_\omega^*$.

The key result about the process with generator $\cA$ is the following.
%
\begin{proposition}
\label{prop-block}
Let $\gap(\cA)$ be the spectral gap of $\cA$ and recall that
$\gap(\cL;q^*)$ denotes the
spectral gap of the East-like process with parameter $q^*$. Then
\[
\gap(\cA)=\gap \bigl(\cL;q^* \bigr).
\]
\end{proposition}
\begin{pf}
Given $\omega\in\O^*=S^{{\mathbb Z}^d}$ consider the new variables
$\eta
_x=0$ if
$\omega_x\in G$
and $\eta_x=1$ otherwise, $x\in{\mathbb Z}^d$. The projection process
on the
$\eta$
variables coincides with the East-like process at
density $p=1-q^*$ because the constraints depend on $\omega$
only through the $\eta$'s. Thus, $\gap(\cA)\le\gap(\cL;q^*)$. To
establish the
converse inequality, we notice that Lemma~\ref{fin-vol} applies as is to
$\cA$. Therefore, it is enough to show that, for any $L$,
$\gap(\cA^\mathrm{ max}_{\L_L})\ge\gap(\cL^\mathrm{ max}_{\L_L};q^*)$ where
$\L_L=[1,L]^d$.

For this purpose,
consider the East-like process in $\L_L$ with maximal boundary
conditions and let $\t_x$ be the first time that there is a legal ring
at the vertex $x\in\L_L$. Using Lemma~\ref{persistence},\hskip.2pt\footnote{Lemma~\ref{persistence} is stated and proved for the whole lattice
${\mathbb Z}^d$ at equilibrium, however, a similar proof applies to the
finite volume
setting at equilibrium for a fixed boundary condition. Furthermore
$\bbP_\h^{\L_L,\s}(\tau_x \geq t)\leq(p \wedge q)^{-L}\bbP_\pi
^{\L_L,\s
}(\tau_x \geq t)$.} we get that, for any $\eta\in\O_{\L_L}$ and any
$x\in\L_L $,
%
\begin{equation}
\label{eq:8} \liminf_{t\to\infty}-\frac{1}{t}\log
\bbP_\eta^{\L_L,\mathrm{ max}}(\t _x\ge t)\ge \gap \bigl(
\cL^\mathrm{ max}_{\L_L};q^* \bigr).
\end{equation}
Let $\O^*_{\L_L} = S^{\L_L}$, then for any $f\dvtx \O^*_{\L_L}\mapsto
\bbR$
with $\mu_{\L_L}(f)=0$ (where $\mu_{\L_L}$ denotes the product measure
on $\O_{\L_L}^*$ with marginal $\mu$ at each site) we now write
%
\begin{eqnarray}
\label{eq:lastday} e^{t\cA_{\L_L}^\mathrm{ max}}f(\omega)&=& \bbE^*_\omega \bigl(f \bigl(\omega
(t) \bigr)\bigr)
\nonumber
\\[-8pt]
\\[-8pt]
\nonumber
&=&\bbE^*_\omega \bigl(f \bigl(\omega(t) \bigr){\mathds
1}_{\{\max_x \t_x <t\}} \bigr) + \bbE^*_\omega \bigl(f \bigl(\omega (t)
\bigr){ \mathds 1}_{\{\max_x \t_x \ge t\}} \bigr),
\end{eqnarray}
where $\bbE^*_\omega(\cdot)$ denotes expectation w.r.t. the chain
generated by $\cA_{\L_L}^\mathrm{ max}$ starting at $t=0$ from
$\omega$. Notice that, for any $x\in\L_L$ and any $t>0$, the event
$\{\t_x\le t\}$ can be read off from the evolution of the projection
variables $\eta$. In particular,
\[
\bbP^*_\omega(\tau_x>t)=\bbP_{\eta(\omega)}^{\L_L,\mathrm{ max}}(
\tau _x >t).
\]
Fix $\varepsilon>0$. Using \eqref{eq:8}, the absolute value of second term
in the RHS of
\eqref{eq:lastday} is
bounded from above by $C\exp\{- t(
\gap(\cL^\mathrm{ max}_{\L_L};q^*)-\varepsilon)\}$ for some constant $C$
depending on $f$
and $L$.

To bound the first term in the
RHS of \eqref{eq:lastday} we observe that, conditionally on the variables
$\eta_x(t)={\mathds1}_{\omega_x\in G}(\omega(t))$ and on the event
$\{\max_{x\in\L_L} \t_x < t\}$, the
variables $\omega_x(t)$ are independent with law $\mu(\cdot\mid G)$ if
$\eta_x(t)=1$ and $\mu(\cdot\mid G^c)$ otherwise. Thus, with
$g(\eta):=\mu_{\L_L}(f\mid\eta)$,
\begin{eqnarray*}
\bbE^*_\omega \bigl(f \bigl(\omega(t) \bigr){\mathds1}_{\{\max_x \t_x <t\}
}
\bigr)&=& \bbE^*_\omega \bigl(g \bigl(\eta(t) \bigr){\mathds1}_{\{\max_x \t_x <t\}
}
\bigr)
\\
&=& \bbE_{\eta(\omega)}^{\L_L,\mathrm{ max}} \bigl(g \bigl(\eta(t) \bigr) \bigr) - \bbE
^*_\omega \bigl(g \bigl(\eta(t) \bigr){\mathds 1}_{\{\max_x \t_x \ge t\}}
\bigr).
\end{eqnarray*}
By construction $\pi(g)=0$, so that
\[
\max_\omega\bigl|\bbE_{\eta(\omega)}^{\L_L,\mathrm{ max}} \bigl(g \bigl(
\eta (t) \bigr) \bigr)\bigr|\le Ce^{-t
\gap(\cL^\mathrm{ max}_{\L_L};q^*)},
\]
and we may bound the term $\bbE^*_\omega (g(\eta(t)){\mathds
1}_{\{\max
_x \t_x \ge t\}} )$ similar to the second term in \eqref
{eq:lastday} using the claim \eqref{eq:8}.
In conclusion,
\[
\max_\omega\bigl| e^{t\cA_{\L_L}^\mathrm{ max}}f(\omega)\bigr|\le C'
e^{-t (
\gap(\cL^\mathrm{ max}_{\L_L};q^*)-\varepsilon)},
\]
so that, by the arbitrariness of $\varepsilon$, $\gap(\cA_{\L_L}^\mathrm{
max})\ge\gap(\cL_{\L_L}^\mathrm{ max};q^*)$.
\end{pf}
A concrete example of the process with generator $\cA$, which will play
a key role in
our proofs, goes as follows.
%
\begin{definition}[(The East-like block process)]
\label{Eblock}
Let
$\L_\ell=[1, \ell]^d$ be the cube of side $\ell$, let
$S= \{0,1\}^{\L_\ell}$, let $\mu=\pi_{\L_\ell}$ and let $G=\{\s
\in S\dvtx
\s_y=0$ for some $y\in\L_\ell\}$. Thus, $q^*=\mu(G)=
1-(1-q)^{\ell^d}$. Let us identify $\omega\in\O^*=S^{{\mathbb
Z}^d}$ with
$\eta\in
\O=\{0,1\}^{{\mathbb Z}^d}$ by setting $\omega_x=\eta_{\L_\ell(x)}$
where $\L_\ell
(x):=\L_\ell+ \ell x$.
Then the process with generator $\cA$
given by \eqref{eq:3bis} associated to the above choice of $\mu,S,G$
corresponds to the following Markov process for $\eta$, called
East-like block process:
the configuration in each block $\L_\ell(x)$, with rate one is replaced
by a fresh one
sampled from $\mu$, provided that, for some $e\in\cB$, the block
$\L_\ell(x-e)$ contains a vacancy.
\end{definition}
The above construction together combined with Proposition~\ref{prop-block}
suggests a possible route, reminiscent of the renormalization group
method in statistical physics, to bound the relaxation time $\trel
({\mathbb Z}
^d;q)$ of the
East-like process.

Using comparison methods for Markov chains \cite{Diaconis}, one may
hope to establish a
bound on $\trel({\mathbb Z}^d;q)$ of the form (cf. Lemma~\ref{blok-versus-ss})
\[
\trel \bigl({\mathbb Z}^d;q \bigr)\le f(q,\ell)\trel(
\cL_\mathrm{ block}),
\]
for some explicit function $f$ where $\cL_\mathrm{ block}$ is the
generator of the East-like block process. Using Proposition~\ref
{prop-block}, one
would then derive the functional inequality
\[
\trel \bigl({\mathbb Z}^d;q \bigr)\le f(q,\ell)\trel \bigl({\mathbb
Z}^d; 1-(1-q)^{\ell
^d} \bigr),
\]
where $\ell$ is a \emph{free} parameter. The final inequality obtained
after optimizing over
the possible choices of $\ell$ would clearly represent a rather
powerful tool.

In order to carry on the above program, we will often use the following
technical ingredient (cf. \cite{CMRT}, Claim~4.6).
%
\begin{lemma}[(The enlargement trick)]
\label{lem:enlarge}Consider two boxes $\L_1=\prod_{i=1}^d[a_i,c_i]$ and
$\L_2=\prod_{i=1}^d[b_i,c_i]$, with $a_i< b_i\le c_i\ \forall i$. Let
$\chi(\eta)$ be the indicator function of the event that the
configuration $\eta\in\O$ has a zero inside the box
$\L_3=\prod_{i=1}^d[a_i,d_i]$ where $a_i\le d_i<b_i, \forall i$. Then
%
\[
\pi \bigl(\chi\var_{\L_2}(f) \bigr)\le\trel^\mathrm{ min}(\L
_1;q)\sum_{x\in
\L_1}\pi
\bigl(c_x\var_x(f) \bigr)\qquad \forall f\in L^2(
\O,\pi).
\]
\end{lemma}
\begin{pf}
For a configuration $\eta$, such that $\chi(\eta)=1$, let
$\xi=(\xi_1,\ldots, \xi_d)$ be the location of the first zero of
$\eta$
in the box $\L_3$
according to the order given by the $\ell_1$ distance $\|\cdot\|_1$ in
${\mathbb Z}^d$ from
the vertex
$a=(a_1,\ldots, a_d)$ of the box $\L_1$ and some arbitrary order on the
hyperplanes $\{y\in{\mathbb Z}^d\dvtx  \|y-a\|_1= \mathrm{const}\}$. Let $\L_\xi
$ be the
box $[\xi_1+1,c_1]\times\prod_{i=2}^d[\xi_i,c_i]$. Then
\begin{eqnarray*}
\pi \bigl(\chi\var_\L(f) \bigr)&=& \sum
_{z\in\L_3}\pi \bigl({\mathds 1}_{\{\xi=z\}}
\var_{\L_1}(f) \bigr)\le\sum_{z\in\L_3}\pi
\bigl({\mathds 1}_{\{\xi=z\}}\var_{\L_{\xi}}(f) \bigr)
\\
&\le&\sum_{z\in\L_3}\trel^\mathrm{ min}(
\L_z;q)\pi \biggl({\mathds 1}_{\{\xi=z\}}(\eta)\sum
_{x\in
\L_z}c_x^{\L_z,\eta\restriction_{\partial_E \L_z}}\var _x(f)
\biggr)
\\
&\le&\trel^\mathrm{ min}(\L_1;q)\sum
_{z\in\L_3} \pi \biggl({\mathds 1}_{\{\xi=z\}}\sum
_{x\in\L_z}c_x \var_x(f) \biggr)
\\
&\le&\trel^\mathrm{ min}(\L_1;q)\sum
_{x\in
\L_1} \pi \bigl(c_x\var_x(f) \bigr).
\end{eqnarray*}
Above we used the convexity of the variance in the first inequality,
the Poincar\'e inequality for the box $\L_z$ together with
Lemma~\ref{lem:ineq} in the second inequality, again
Lemma~\ref{lem:ineq} together with the equality
$c_x^{\L,\eta\restriction_{\partial_E\L}}(\eta)=c_x(\eta)$ for all
$\L,x\in\L$ and $\eta\in\O$.
\end{pf}

\subsection{Capacity methods}
Since the East-like process in a box $\L\subset
{\mathbb Z}^d$ with boundary conditions $\s$ has a reversible measure (the
measure $\pi_\L$), one can associate
to it an
\emph{electrical network} in the standard way
(cf., e.g., \cite{G09}). For lightness of notation in what follows, we
will often drop the
dependence on $\L$ of the various quantities of interest.

We first define the transition rate $ \cK^\s(\h,\h') $ between two
states $\h,\h'\in\O_\L$ as
\begin{eqnarray*}
\cK^\s \bigl(\h,\h' \bigr) = \cases{
c_x^{\L,\s}(\h) \bigl[q \h_x+ p(1-
\h_x) \bigr], &\quad  $\mbox{if $\h' = \h^x$ for
some $x\in\L$},$ \vspace*{2pt}
\cr
0, & \quad $\mbox{otherwise.}$ }
\end{eqnarray*}
Since the process is reversible, we may associate with each pair $(\h
,\h
')\in\O_{\L}^2$ a~conductance $\cC^\s(\h,\h')=\cC^\s(\h',\h)$
in the usual way [see \eqref{eq:gen}],
%
\begin{equation}
\cC^\s \bigl(\h,\h' \bigr) =\pi(\h)
\cK^\s \bigl(\h,\h' \bigr).
\end{equation}
Observe that $\cC^\s(\h,\h') > 0$ if and only if $\h' = \h^x$ for some
$x\in\L$ and $c_x^{\L,\s}(\h) = 1$.
We define the edge set of the electrical network by
\[
E_\L^\s= \bigl\{ \bigl\{\h,\h' \bigr\}
\subset\O_\L\dvtx \cC^\s \bigl(\h,\h'
\bigr) > 0 \bigr\}.
\]
Notice that $E_\L^\s$ consists of unordered pairs of configurations. We
define the resistance $r^\s(\h,\h')$ of the edge $\{\h,\h'\}\in
E_\L^\s
$ as the reciprocal of the conductance $\cC^\s(\h,\h')$.
With the above notation, we may express the generator \eqref{eq:gen} as
\[
\cL_\L^\mathrm{\s}f(\h) = \sum
_{x\in\L} \frac{\cC^\s(\h,\h
^x)}{\p(\h
)} \bigl[ f \bigl(\h^x
\bigr)-f(\h) \bigr].
\]
Given $B \subset\O_\L$, we denote by $\t_B$ the hitting time
\[
\t_B=\inf \bigl\{ t > 0 \dvtx \eta(t) \in B \bigr\},
\]
and denote by $\t_B^+$ the first return time to $B$
\[
\t_B^+=\inf \bigl\{ t > 0 \dvtx \eta(t) \in B, \h(s)\neq\h(0)
\mbox{ for some } 0<s<t \bigr\}.
\]
We define the \emph{capacity} $C^\s_{A,B}$ between two disjoint subsets
$A$, $B$ of $\O_\L$ by
%
\begin{equation}
\label{eq:cap-def} C_{A,B}^\s= \sum
_{\z\in A}\pi(\z)\cK^\s(\z) \bbP^{\L,\s}_\z
\bigl(\tau _A^+ > \tau_B \bigr) ,
\end{equation}
where $\cK^\s(\z) = \sum_{\xi\neq\z}\cK^\s(\z,\xi)$ is the
holding rate
of state $\z$ (see, e.g., \cite{Beltran}, Section~2).
The \emph{resistance} between two disjoint sets $A,B$ is defined by
%
\begin{equation}
\label{eq:res-def} R^\s_{A,B} := 1/C_{A,B}^\s.
\end{equation}
With slight abuse of notation, we write $C_{\z,B}^\s$ and $R_{\z,B}$,
if $A=\{\z\}$ with $\z\notin B$.
The mean hitting time $\bbE_\z^{\L,\s}(\t_B)$ can be
expressed as (see, e.g., formula (3.22) in \cite{BovierNew}):
%
\begin{equation}
\label{eq:Thit-Cap} \bbE^{\L,\s}_{\z}(\t_B)=
R_{\z,B}^\s\sum_{\h\notin B} \p(\h)
\bbP^{\L
,\s}_\h (\tau_{\{\z\}} < \tau_B
).
\end{equation}
The following variation principle, useful for finding \emph{lower}
bounds on the resistance (i.e., upper bounds on the
capacity), is known
as the \emph{Dirichlet principle} (see, e.g., \cite{G09}):
%
\begin{equation}
\label{eq:dir-princ} C_{A,B}^\s= \inf \bigl\{
\cD_\L^\s(f) \dvtx f\dvtx \O_\L\to\bbR,
f|_A= 1 , f|_{B}=0 \bigr\},
\end{equation}
where the Dirichlet form $\cD_\L^\s(f)$ is given in \eqref{eq:DirForm}.
%
\begin{remark}
\label{rem:capmono}
It is clear from \eqref{eq:DirForm} that the capacity increases as
vacancies are added to the boundary conditions and, therefore, the
resistance decreases.
This is also a consequence of the Rayleigh's monotonicity principle,
which states that inhibiting allowable transitions of the process can
only increase the resistance.
\end{remark}
In order to get \emph{upper} bounds on the resistance, it is useful to
introduce the notion of a \emph{flow} on the electrical network.
For this purpose, we define the set of oriented edges
\[
\vec{E}_\L^\s= \bigl\{ \bigl(\h,\h'
\bigr) \in\O_\L^2 \dvtx \bigl\{\h,\h' \bigr
\} \in E_\L ^\s \bigr\}.
\]
For any real valued function $\theta$ on
\emph{oriented edges}, we define the divergence of $\theta$ at $\xi
\in\O
_\L$ by
\[
\div\theta(\xi) = \sum_{\h \dvtx (\xi,\eta)\in\vec{E}_\L^\s
}\theta (\xi,\h).
\]
%
\begin{definition}[(Flow from $A$ to $B$)]
A flow from the set $A\subset\Omega_\L$ to a disjoint set $B\subset
\Omega_\L$, is a real
valued function $\theta$ on $\vec E_\L^\s$ that is antisymmetric
[i.e.,
$\theta(\s,\h)=-\theta(\h,\s)$] and satisfies
\begin{eqnarray*}
\div\theta(\xi) &=& 0 \qquad\mbox{if } \xi\notin A\cup B ,
\\
\div\theta(\xi) &\geq&0\qquad \mbox{if } \xi\in A ,
\\
\div\theta(\xi) &\leq&0\qquad \mbox{if } \xi\in B.
\end{eqnarray*}
The strength of the flow is defined as $|\theta| = \sum_{\xi\in
A}\div\theta(\xi)$. If $|\theta|=1$ we call $\theta$ a \emph{unit flow}.
\end{definition}
%
\begin{definition}[(The energy of a flow)] The energy associated with a
flow $\theta$ is defined by
%
\begin{equation}
\label{eq:flowEnergy} \cE(\theta) = \frac{1}2 \sum
_{(\h,\h')\in\vec E_\L^\s} r^\s \bigl(\h ,\h' \bigr)
\theta \bigl(\h,\h' \bigr)^2.
\end{equation}
\end{definition}
With the above notation, \emph{Thompson's principle} states that
%
\begin{equation}
\label{eq:thom} R_{A,B}^\s= \inf \bigl\{ \cE(\theta) \dvtx
\theta\mbox{ is a unit flow from } A \mbox{ to } B \bigr\},
\end{equation}
and that the infimum is attained by a unique minimizer called the \emph
{equilibrium flow}.

We conclude with a concrete application to the East-like
process. Given $x\in\L=[1,L]^d$, let $\tau_x$ be the
hitting time of the set $B_x:=\{\h\in\O_\L\dvtx \h_x=0\}$ for the
East-like process in $\L$ with $\s$ boundary
condition and let
$\bbE_{\mathds1}^{\L,\s}(\t_x)$ be its average when the starting
configuration has no
vacancies (denoted by $\mathds1$).
Also, let $\tilde\tau_x$ be the
hitting time of the set $\widetilde B_x:=\{\h\in\O_\L\dvtx \h_x=1\}$ for
the East-like process in $\L$ with $\s$ boundary
condition and let $\bbE_{\mathds10}^{\L,\s}(\tilde\t_x)$ be its
average when the starting configuration has only a single vacancy which
is located at $x$ (denoted by $\mathds10$).
%
\begin{lemma}\label{lem:capbnds}
Suppose $L\leq2^{\theta_q/d }$ with $q < 1/2$.
Then there exists a constant $c > 0$ (independent from $q$ and $d$)
such that
%
\begin{eqnarray}
\label{eq:cap-bnds} c R_{\1,B_x}^\s&\leq&\bbE_{\mathds1}^{\L,\s}(
\t_x)\leq R_{\1
,B_x}^\s \quad \mbox{and}
\nonumber
\\[-8pt]
\\[-8pt]
\nonumber
c q R_{\1,\widetilde B_x}^\s&\leq&\bbE_{\mathds10}^{\L,\s}(
\tilde \t _x)\leq q R_{\1,\widetilde B_x}^\s.
\end{eqnarray}
\end{lemma}
\begin{pf}
Setting $c := \inf\{ (1-q)^{1/q} \dvtx q \in(0,1/2)\}>0$, we have
$(1-q)^{L^d} \geq c$.
We now observe that
\[
c \leq(1-q)^{L^d} \leq\pi(\1) \leq\sum_{\h\in B_x^c}
\pi(\h)\bbP ^{\L,\s}_{\h}(\tau_{\1}<
\t_{x})\leq\pi \bigl(B^c_x \bigr) = p\leq1 ,
\]
and similarly
\[
c q \leq q(1-q)^{L^d} \leq\pi(\10) \leq\sum
_{\h\in\widetilde
B_x^c}\pi(\h)\bbP^{\L,\s}_{\h}(
\tau_{\10}<\widetilde\t _{x})\leq\pi \bigl(\widetilde
B^c_x \bigr) = q ,
\]
the result follows at once from \eqref{eq:res-def} and \eqref{eq:Thit-Cap}.
\end{pf}

\subsection{Bottleneck inequality}
\label{bottleneck}
One can lower bound the relaxation time (i.e., upper bound the
spectral gap) by restricting the variational formula \eqref{eq:gap}
to indicator functions of subsets of $\O_\L$. In this way, one gets
(cf., e.g., \cite{Saloff})
%
\begin{equation}
\label{biancaneve} T^{\s} _\mathrm{ rel}(\L;q) \geq\max
_{A\subset\O_\L}\frac{ \p(A)
\p
(A^c) }{\cD^{\s}_\L( \mathds{1}_A) }.
\end{equation}
Using reversibility, the Dirichlet form $\cD^{\s}_\L(\mathds{1}_A)$ can
be written as
%
\begin{equation}
\label{rane} \cD^{\s}_\L(\mathds{1}_A)=
\sum_{\eta\in A,\h'\in A^c}\pi(\eta )\cK^{\s
} \bigl(\eta,
\h' \bigr)= 
\sum
_{\eta\in\partial A} \p(\eta) \cK^{\s} \bigl(\eta,A^c
\bigr), 
\end{equation}
where
%
\begin{equation}
\label{a-} \partial A:= \bigl\{ \h\in A \dvtx \exists \h'\in
A^c \mbox{ such that } \cK^{\s} \bigl(\h,\h'
\bigr)>0 \bigr\}
\end{equation}
is the internal boundary of $A$ and
%
\begin{equation}
\label{atmosfera} \cK^{\s} \bigl(\eta,A^c \bigr) = \sum
_{\s\in A^c}\cK^{\s}(\eta,\s)
=
\mathop{\sum_{x\in\L\dvtx  c^{\L, {\s}}_x(\eta)=1 ,}}_{ \eta^x \notin
A} \bigl\{q
\eta_x +p(1-\eta_x) \bigr\}
\end{equation}
is the escape rate from $A$ when the chain is in $\eta$.
Using the trivial bound $\cK^{\s} (\eta,A^c)\le L^d$, we get that
$\cD
^{\s}_\L(\mathds{1}_A)\le
L^d \pi(\partial A)$ and the relaxation time satisfies
\[
\trel^\mathrm{ max}(\L;q)\ge\max_{A\subset\O_\L}
\frac
{1}{L^{d}} \frac{ \p
(A) \p(A^c) }{\pi(\partial A)}.
\]
The boundary $\partial A$ of a set $A$ with a small ratio $\cD^{\s
}_\L
( \mathds{1}_A) /
 (\p(A) \p(A^c) )$ is usually referred to as a \emph{bottleneck}.
A good general strategy to find lower bounds on the relaxation time is
therefore to look for small bottlenecks
in the state space (cf. \cite{Levin2008,Saloff}).
\section{Algorithmic construction of an efficient bottleneck}
\label{bottiglia}
In this section, we will construct a bottleneck (cf. Section~\ref{bottleneck}) which will prove some of the lower bounds in Theorems
\ref
{th:main2} and \ref{th:main3}.
%
\begin{theorem}
\label{pupi}
Fix $\L=[1,L]^d$, with $L=2^n$ and $n\le\theta_q/d$. Then there exists
$A_*\subset\O_\L$ such that
%
\begin{equation}
\label{eq:dirbndbis} \cD^\mathrm{ max}_\L(\1_{A_*})
\leq2^{-n\theta_q +d
{n\choose2}-n\log_2 n +O(\theta_q) },
\end{equation}
and $1/2>\pi(A_*) \geq q/2$ for $q$ sufficiently small.
\end{theorem}
In the one-dimensional case, the construction of a bottleneck with the above
properties has been carried out in \cite{CFM}. The extension to higher
dimensions requires some nontrivial
generalization of the main ideas of \cite{CFM} and the whole analysis
of the bottleneck $A_*$ becomes more
involved. The plan of the proof goes as follows:
\begin{longlist}[1.]
\item[1.] We first define the set $A_*$. For any $\eta\in\O_\L$, we
will remove its vacancies according to a deterministic rule until we either
reach the configuration without vacancies, and in that case we say
that $\eta\notin A_*$, or we reach the
configuration with exactly one vacancy at the upper corner $v^*$ of
$\L$ and in that case we declare $\eta\in A_*$.
\item[2.] Next, we prove some structural properties of the
configurations $\eta\in\partial A_*$. The main combinatorial result
here is that,
if $L=2^n$, then
$\eta\in\partial A_*$ \emph{must} have at least $n+1$ ``special''
vacancies at
vertices $(z_1,\ldots,z_{n+1})$, where the range of the possible
values of the $(n+1)$-tupla $(z_1,\ldots,z_{n+1})$ is a set
$\G_\L^{(n)}$ of
cardinality $|\G_\L^{(n)}|\le2^{2d(n+1)}
\frac{ 2^{ d {n\choose2} } }{n! d^n}$.
\item[3.] The proof is readily finished by observing that $\pi
(\partial
A_*)\le q^{n+1}|\G_\L^{(n)}|$.
\end{longlist}

\subsection{Construction of the bottleneck}
In what follows, we write $\ll$ for the lexicographic order in
${\mathbb Z}
^d$, that is, $x \ll y$
if and only if $x_i \leq y_i$ for all $1\leq i \leq d$. For any box
$\L$, we define $\bar\L:=\L\cup\partial_E \L$.
Given $\eta\in\O_\L$, with some abuse of notation we will sometimes
also denote by
$\eta$ the configuration in $\O_{\bar\L}$ which coincides with
$\eta$
on $\L$ and which
is zero on $\partial_E\L$.
%
\begin{definition}
Given $x \in\L$ and $\eta\in\O_\L$, we define the gap of $x$ in
$\eta
$ by
%
\begin{equation}
\label{saab} g_x (\eta):= \min\bigl\{ g >0\dvtx \exists z \in\bar\L
\mbox{ with } z \ll x , \eta_z=0 , \|x-z\|_1=g\bigr\}.
\end{equation}
If $\eta_x=0$, we say that $g_x(\eta)$ is the gap of the vacancy at $x$.
\end{definition}
Note that in \eqref{saab}
$g$ varies among the positive integers and the minimum is always
realized since $\eta$ is defined to be zero on $\partial_E \L$.
Moreover, we know that $g_x (\eta) \leq L$. See Figure~\ref{fig:grappino} for an example.
%
\begin{figure}[b]

\includegraphics{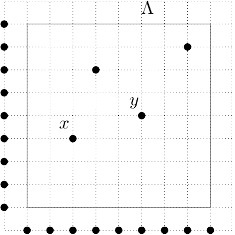}

\caption{A configuration $\eta$ extended with maximal boundary
conditions. The gap of the vacancies at $x,y$ are: $g_x(\eta)=3$,
$g_y(\eta)=4$.}
\label{fig:grappino}
\end{figure}

Following \cite{CFM}, we now define a deterministic discrete time
dynamics, which will be the key input for the construction of the
bottleneck $A_*$.

Starting from $\eta$, the successive stages of the dynamics will be
obtained recursively by first removing from $\eta$ all vacancies with
gap one, then removing from the resulting configuration all
vacancies with gap two, and so on until all vacancies with gap
size $L-1$ have been removed. We stop before removing all vacancies
with gap $L$ since this would always give rise to the configuration
with no vacancy.

More precisely, given $\eta\in\O_\L$ and a positive integer $g$, we
define $\phi_g(\eta)\in\O_\L$ as
%
\begin{equation}
\phi_{g} (\eta)_y:= \cases{ 1, &\quad  $\mbox{ if }
g_y(\eta)=g ,$\vspace*{2pt}
\cr
\eta_y, &\quad $\mbox{
otherwise}.$}
\end{equation}
Then the deterministic dynamics starting from $\eta$ is given by the trajectory
$(\Phi_0 (\eta)$, $\Phi_1 (\eta),\ldots,  \Phi_{L-1}(\eta))$ where
\[
\Phi_0 (\eta):= \eta,\qquad \Phi_g(\eta):=
\phi_g \bigl( \Phi_{g-1} (\eta) \bigr),\qquad g=1,2, \ldots, L-1.
\]
Since all vacancies in $\Phi_{L-1}(\eta)$ have gap of size at least
$L$, the configuration $\Phi_{L-1}(\eta)$ can either be the
configuration with no vacancies, in the sequel denoted by
$\mathds{1}$, or the configuration with exactly one vacancy at $v^*:=
(L,L,\ldots,L)$, in the sequel denoted by $\mathds{1}0$. In what
follows, it will be convenient to say that a vacancy at vertex $x$ is
removed at stage $g$ from a configuration $\z$
if $\Phi_{g-1} (\z)_{x}=0$ and
$\Phi_{g} (\z)_{x}=1$.

We are now in a position to define the bottleneck.
%
\begin{definition}
We define $A_* \subset\O_\L$ as the set of configurations $\eta\in
\O
_\L$ such that $\Phi_{L-1} (\eta)=\mathds{1}0$.
\end{definition}
%
\begin{remark}
Since $ \mathds{1}0 \in A_*$ and $\mathds{1}\notin A_*$,
any path in $\O_\L$ connecting
$ \mathds{1}0$ to $\mathds{1}$ (under the East-like dynamics with
maximal boundary conditions) must cross $\partial A_*$ [cf. \eqref{a-}].
\end{remark}
Some properties of the deterministic dynamics, which are an
immediate consequence of the definition and are analogous to those
already proved in \cite{CFM} for the
one-dimensional case, are collected below.
\begin{itemize}
\item The deterministic dynamics only remove vacancies so gaps are
increasing under the dynamics. Also, if $\h_x=0$ and $g_x(\h)=g$, then
$\Phi_{d}(\h)_x=0$ for all $d<g$.
\item$\Phi_g(\h)$ contains no vacancies with gaps smaller or equal
to $g$.
\item Whether the deterministic dynamics remove a vacancy at a point
$x$ depends only on $\{y\in\bar\L| y \ll x , y \neq x\}$.
\item For two initial configurations $\h$ and $\h'$, if $\Phi_g(\h
)=\Phi
_g(\h')$ then $\Phi_m(\h)=\Phi_m(\h')$ for all $m \geq g$. In this
case, we say the configurations $\h$ and $\h'$ are \emph{coupled} at
gap~$g$.
\end{itemize}
%
\subsection{Some structural properties of \texorpdfstring{$\partial A_*$}{$partial A_*$}}
Analogously to the
one-dimen\-sional case (cf. \cite{CFM}, Lemma~5.11), in order to compute
the cardinality of the boundary of the bottleneck
$\partial A_*$, we need to prove a structural result for the configurations
in $\partial A_*$.

Given $\h\in\partial A_*$ and $z\in\L$ such that $c_z^{ \L, \mathrm{
max}}(\h)=1$ and $\h^z\notin A_*$,
we know that at each stage $g$ of the deterministic dynamics there
must be at least one vertex at which the two configurations $\Phi_g(\h)$
and $\Phi_g(\h^z)$ differ. Furthermore, at least one of these
discrepancies must give rise to a new discrepancy before it is
removed by the deterministic
dynamics, and this must continue until
$\Phi_{L-1}(\h),\Phi_{L-1}(\h^z)$ have a discrepancy at the vertex
$v^*$. The next lemma clarifies this mechanism.
%
\begin{lemma}\label{lem:prelim}
Let $\eta\in\partial A_*$ and $z \in\L\setminus\{v^*\}$ be such
that $c_z ^{\L, \mathrm{ max}}(\eta)=1$ and
$\eta^z \notin A_*$. Then there exists a sequence
$z = u_0 \ll u_{1} \ll\cdots\ll u_M = v^*\in\L$ of length $M\geq
1$ such that, if 
$d_i:=\|u_{i-1}-u_{i}\|_1$, then $1=d_1<d_2<\cdots<d_M<L$ and:
\begin{longlist}[(iii)]
\item[(i)] $\Phi_{\ell}(\h)_{u_{i}}=\Phi_{\ell}(\h^z)_{u_{i}}=0$ for
$\ell
< d_{i}$,
\item[(ii)] $\Phi_{d_i}(\h)_{u_{i}}\neq\Phi_{d_i}(\h^z)_{u_{i}}$,
\item[(iii)] $\Phi_{d_i-1}(\h)_{u_{i-1}}\neq\Phi_{d_i-1}(\h
^z)_{u_{i-1}}$.
\end{longlist}
\end{lemma}
The above properties can be described as follows. Both $\eta$ and
$\eta^z$ have a vacancy at $u_i, i=1,\ldots,M$. The vacancy
at $u_i$ survives for both configurations up to and including the
$(d_i-1)$th stage of
the deterministic dynamics. At stage $d_i$, the vacancy at $u_i$ is
removed from one configuration but not from the other because of the
original vacancy at $u_{i-1}$. The latter, in fact, is at distance
$d_i$ from $u_i$ and it survives up to the $(d_i-1)$th stage of
the deterministic dynamics only in one of the two configurations
$\Phi_{d_i-1}(\h), \Phi_{d_i-1}(\h^z)$.
%
\begin{pf*}{Proof of Lemma \ref{lem:prelim}}
We proceed by induction from $v^*$ toward $z$.
This gives rise to a sequence of vertices $\{v_i\}_{i=1}^{M+1}$ and
distances $\{c_i\}_{i=1}^M$ from which we define $\{u_i\}_{i=0}^{M}$,
$\{d_i\}_{i=1}^{M}$ by $u_i= v_{M-i+1}$ and $d_i=c_{M-i+1}$.

We begin by setting $v_1=v^*$. Since $\h\in\partial A_*$ and
$\h^z\notin A_*$, it follows that $\Phi_{\ell}(\h)_{v_1}=0$ for all
$\ell< L$ and $\Phi_{L-1}(\h^z)_{v_1}=1$.
Thus, there exists $1\leq c_1\leq L-1$ such that the vacancy at $v_1$
is removed from $\h^z$ but not from $\eta$ at stage $c_1$ of the
dynamics. 
This implies, in particular, that $g_{v_1}(\Phi_{c_1-1}(\h^z))=c_1$,
so that there exists a $v_{2} \ll v_1$ such that $\|v_1-v_{2}\|_1=c_1$
and $\Phi_{c_1-1}(\h^z)_{v_2}=0$.
Using the fact that the vacancy at $v_1$ is not removed from $\eta$ at
the $c_1$-stage of the dynamics, we conclude that the vacancy at
$v_{2}$ cannot be present in $\Phi_{c_1-1}(\eta)$, that is, $\Phi
_{c_1-1}(\h)_{v_2}=1$.
Since the deterministic dynamics at the point $v_2$ depend only on the
initial configuration in the region $\{x \in\L\dvtx x \ll v_2\}$, we
must have $z \ll v_2$.
This completes the proof of the first inductive step (note that the
proof is complete if $z= v_2$).

Assume now inductively that we have been able to find a sequence $z\ll
v_{k+1} \ll
v_{k} \ll\cdots\ll v_1 = v^*\in\L$ such that, setting $c_i := \|
v_i-v_{i+1}\|_1$, for $1 \leq i \leq k$ the following holds; $ 1\leq
c_k < c_{k-1} < \cdots< c_2 < c_1 <L$ and:
\begin{longlist}[(a)]
\item[(a)]$ \Phi_\ell(\h)_{v_i}=\Phi_\ell(\h^z)_{v_i}=0, \forall
\ell<c_i$,
\item[(b)] $ \Phi_{c_i}(\h)_{v_i}\neq\Phi_{c_i}(\h^z)_{v_i}$,
\item[(c)] $ \Phi_{c_i-1}(\h)_{v_{i+1}}\neq\Phi_{c_i-1}(\h^z)_{v_{i+1}}$.
\end{longlist}
If $c_k=1$, then $\h_{v_{k+1}} = \Phi_0(\h)_{v_{k+1}}\neq
\Phi_0(\h^z)_{v_{k+1}} =\h^z_{v_{k+1}}$ which in turn implies
$v_{k+1}=z$ and we stop, and fix $M=k$.
Otherwise, we may repeat the argument used for the first step as
follows.

The equality $\Phi_{c_k-1}(\h)_{v_{k+1}}\neq\Phi_{c_k-1}(\h
^z)_{v_{k+1}}$ implies that there must exist a first stage $c_{k+1}\leq
c_{k} - 1$
at which $\Phi$ removes the vacancy at $v_{k+1}$ from
either $\h$ or $\h^z$ but not from both.
If $c_{k+1}=0$, then again $v_{k+1}=z$, and since $c_z^{\L,\mathrm{ max}}(\h)
= 1$ we have $\Phi_1(\h)_z=\Phi_1(\h^z)_z=1$.
In particular, (c) above with $i=k$ implies that $c_k=1$ and we are
in the case described above, so we set $M=k$ and stop.
Thus, we can assume $c_{k+1} \geq1$.
Then $\Phi_{\ell}(\h)_{v_{k+1}}= \Phi_{\ell}(\h^z)_{v_{k+1}}=0$
for $\ell< c_{k+1}$,
and $\Phi_{c_{k+1}}(\h)_{v_{k+1}}\neq\Phi_{c_{k+1}}(\h^z)_{v_{k+1}}$
[thus assuring (a) and (b) for $i=k+1$]. 
Let $\xi= \h$ if \mbox{$\Phi_{c_{k+1}}(\h)_{v_{k+1}}=1$} and $\xi=\h^z$
otherwise.
So $g_{v_{k+1}}(\Phi_{c_{k+1}-1}(\xi)) = c_{k+1}$ by definition,
which implies that there exists a
$v_{k+2}\ll v_{k+1}$ with $\|v_{k+1}-v_{k+2}\|_1=c_{k+1}$
and $\Phi_{c_{k+1}-1}(\xi)_{v_{k+2}}=0$. Since the vacancy at
$v_{k+1}$ is not removed from $\xi^z$ at stage $c_{k+1}$ of the
dynamics, we must have $\Phi_{c_{k+1}-1}(\xi^z)_{v_{k+2}}=1$ [thus
completing the proof of (c) for $i=k+1$].
Following the same argument as for the first step of the induction we
must also have $z \ll v_{k+2}$.
We may continue by induction until $v_{k+1} = z$ and fix $M=k \geq1$.
The proof now follows by letting $u_i= v_{M-i+1}$ and $d_i=c_{M-i+1}$.
\end{pf*}

In light of the previous technical lemma, we are able to
generalize \cite{CFM}, Lemma~5.11, to higher dimensions.
%
\begin{lemma}\label{svegliadalle3}
Let $\eta\in\partial A_*$ and $z \in\L\setminus\{v^*\}$ be such
that \mbox{$c_z ^{\L, \mathrm{ max}}(\eta)=1$} and
$\eta^z \notin A_*$.
Fix a sequence $(u_i)_{i=0}^M\in\L$ according to
Lemma~\ref{lem:prelim}.
Let $B= \prod_{j=1}^d[a_j,b_j]$ be a box such that \textup{(i)} $B \subset
\bar\L$, \textup{(ii)} $z \in B$, \textup{(iii)} $z \neq a:=(a_1,a_2, \ldots, a_d)$ and
$\eta_{a}=0$, \textup{(iv)} $b:=(b_1,b_2, \ldots, b_d) = u_{k}$ for some
$k \dvtx 0\leq k \leq M$. Let $\ell:= \|B\|_1$ and let
\begin{eqnarray*}
B^- &:=& \bigl\{x \in\bar\L\setminus B \dvtx x \ll a \mbox{ and } \|x-a
\|_1< \ell\bigr\} ,
\\
B^+&:=& \bigl\{ x \in\bar\L\setminus B \dvtx x \gg b \mbox{ and } \|x-b
\|_1 \leq\ell\bigr\}.
\end{eqnarray*}
Then at least one of the following properties is fulfilled:
\begin{longlist}[1.]
\item[1.]$B$ contains $v^*=(L,L, \ldots, L)$ and some point of $\partial_E
\L$.
\item[2.]$a \notin\partial_E \L$ and $\eta$ has at least one vacancy in
$B^-\neq \varnothing$.
\item[3.]$k\neq M$ and $B^+$ contains $u_{k+1}$.
\end{longlist}
\end{lemma}
\begin{pf}
Fix $\h\in\partial A_*$, $z$ and $B$ satisfying the conditions of the lemma.
Suppose first that $a \in\partial_E \L$. If $k=M$, then
$v^*=u_M=u_k=b \in B$, thus implying the thesis. Suppose $k<M$. We know
that $\Phi_\ell( \eta)_b= \Phi_\ell(\eta^z)_b=1$ since $a \in
\partial
_E \L$. Hence, by property (iii) in Lemma~\ref{lem:prelim} (with
$i=k+1$), we have $\ell\geq d_{k+1}$. This implies that $u_{k+1} \in
B^+$, which gives rise to the thesis. Suppose now, for contradiction,
that $a\notin\partial_E \L$ and that the thesis is false, so we have:
\begin{longlist}[($\lnot1$)]
%
\item[($\lnot1$)] $v^*\notin B$ or $B \cap\partial_E \Lambda=
\varnothing$,
\item[($\lnot2$)] $\h_y = 1$ for all $y\in B^-$ (including the case
$B^-= \varnothing$),
\item[($\lnot3$)] either $k=M$ or ``$k\neq M$ and $u_{k+1}\notin B^+$.''
\end{longlist}
We will prove that these assumptions give rise to a contradiction with
the definitions of $(u_i)_{i=0}^M$, $(d_i)_{i=1}^M$.
Note that ($\lnot1$) holds since we assume $a \notin\partial_E \L$
so $B \cap\partial_E \Lambda= \varnothing$.

%
\begin{figure}

\includegraphics{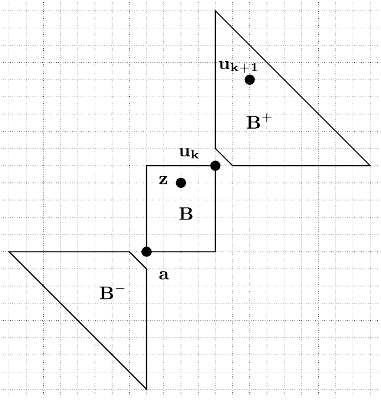}

\caption{The sets $B^-, B^+$ when far from the border of $\L$. We have
considered the case $u_{k+1} \in B^+$. In this example $\ell=9$.}
\label{fig:speriamo}
\end{figure}

First, we claim that the assumption $a\notin\partial_E \L$ together
with ($\lnot2$) implies $\ell< L$. To prove the claim, suppose that
$\ell\geq L$. Since $a \in\bar\L$ and $a\neq v^*$, we know that at
least one coordinate of $a$ is strictly less than $L$, so there exists
a $j\in\{1,\ldots,d\}$ such that $a_j < L$.
Now the point $r=(r_1,\ldots,r_d)$ defined by $r_i = a_i$ for all
$i\neq
j$ and $r_j = 0$ belongs to the East boundary $\partial_E \L$. Also, $r
\ll a$ and $\|r-a\|_1 = a_j < L \leq\ell$, so $r\in B^-$. This
contradicts assumption ($\lnot2$) above.

By ($\lnot2$) $\h_y = 1$ for all $y\in B^-$, so we have
$\Phi_{\ell-1}(\h)_a=\Phi_{\ell-1}(\h^z)_a=0$.
In particular, $g_b(\Phi_{\ell-1}(\h)),g_b(\Phi_{\ell-1}(\h
^z))\leq
\ell$ so $\Phi_{\ell}(\h)_b=\Phi_{\ell}(\h^z)_b=1$. This implies that
$\Phi_{L-1}(\h)_b=1$ hence $b \neq v^*$ (since $\h\in\partial A_*$).
So $b=u_k$ for some $k<M$ and there exists $u_{k+1}\gg u_k$ satisfying
$\|b-u_{k+1}\|_1=d_{k+1}<L$.
By ($\lnot3$) $u_{k+1}\notin B^+$ so that $d_{k+1}>\ell$, so by
monotonicity of the deterministic dynamics $\Phi_{d_{k+1}-1}(\eta)_b
\geq\Phi_\ell(\eta)_b =1$ and $\Phi_{d_{k+1}-1}(\eta^z)_b \geq
\Phi
_\ell(\eta^z)_b =1$. This implies that $\Phi_{d_{k+1}-1}(\eta^z)_b
=\Phi
_{d_{k+1}-1}(\eta)_b = 1$ which contradicts $b=u_k$ [see Lemma~\ref
{lem:prelim}(iii)].
\end{pf}
For any configuration $\eta\in\partial A_*$ the above Lemma~\ref{svegliadalle3} allows us to isolate a special subset of vacancies
of $\eta$. This special subset, in the sequel denoted by $\{
z_1,z_2,\ldots,z_S\}$,
will be defined iteratively by means of an algorithm which we now
describe. In what follows, it will be convenient to use the following
notation: given a box $\L$ and a site $x \in{\mathbb Z}^d \setminus
\L$, we
define $\L\star x$ as the minimal box containing both $\L$ and $x$.
The input of the algorithm is a pair $(\eta,z_0)$, where $\h
\in\partial A_*$ and $z_0 \in\L$ is such that $c^{\L, \mathrm {max}
}_{z_0}(\eta)=1$ and $\eta^{z_0} \notin A_*$. The output will be a
sequence $\{(z_i,\D_i)\}_{i=1}^S$, $S\ge n+1$ if the box $\L=[1,L]^d$
has side
$L=2^n$,
where $\{\D_i\}_{i=1}^S$ is a increasing sequence of boxes contained in
$\bar\L$
and $\{z_i\}_{i=1}^S\subset\bar\L$ contains exactly $S-1$ points in
$\L$ where $\eta$ is zero.
%
\begin{remark}
Necessarily $z_0\neq v^*$. Otherwise, the condition $c^{\L, \mathrm{ max}
}_{z_0}(\eta)=1$ would imply that the gap of the vacancy at $v^*$ is equal
to one and the latter would be removed at the first step of the
deterministic dynamics defining $A_*$. That would contradict the
property $\Phi_{L-1}(\eta)_{v^*}=0$.
\end{remark}
\begin{longlist}[(i)]
\item[\textit{Initial step}.]
Choose an arbitrary sequence of vertices $u_1,\ldots,u_M$ satisfying the
properties
described in Lemma~\ref{lem:prelim} for the pair $(\eta,z_0)$. Define
also $z_1$ to be the minimal element (in lexicographic order)
of the nonempty set $\{ z_0-e \dvtx \eta_{z_0-e} =0 , e \in\cB
\}$ and set $\D_1= \{z_0\} \star z_1$.
\item[\textit{The recursive step}.]
Suppose that $(z_1,\D_1), (z_2,\D_2), \ldots, (z_i,\D_i)$ has been
defined in such a
way that:
\begin{itemize}
\item for all $j\le i$, the set $\D_j$ is a box satisfying: (i)
$\D_j\subset\bar\L$, (ii) $z_0\in\D_j$ but it does not coincides
with the lower corner of $\D_j$ where $\eta$ has a vacancy, (iii) the
upper corner of $\D_j$
coincides with $u_{k_j}$ for some $k_j\in\{0,1,\ldots,M\}$.
\item$z_k\neq z_j$ for all $j\neq k$ and $\eta_{z_j}=0$ for all $j\le i$.
\end{itemize}
Let $\D_i^\pm$ be the two sets
defined in Lemma~\ref{svegliadalle3} for the box $\D_i$ and adopt the
convention that $\{u_{M+1}\}:= \varnothing$. \vskip0.3cm
\begin{itemize}
\item\textit{If} the upper corner of $\D_i$ is $v^*$
and the lower corner of $\D_i$ belongs to $\partial_E \L$ then
\textit{stop};
\item\textit{else}
\begin{itemize}
\item\textit{if} the lower corner of $\D_i$ is not in $\partial_E\L$,
define $z_{i+1}$ to be the
minimal element (in lexicographic order)
of the nonempty set $\{z \in\D_i^- \cup (\D_i^+\cap\{
u_{k_i+1} \}
 ) \dvtx \eta_z=0\}$ and set $\D_{i+1}:= \D_i \star z_{i+1}$;
\item\textit{else} define $z_{i+1}=u_{k_i+1}$ and set $\D_{i+1}:= \D_i
\star
u_{k_i+1}$;
\end{itemize}
\item\textit{Endif}
\end{itemize}
\end{longlist}
%
\begin{remark}
Note that in last case (i.e., upper corner $ \neq v^*$ and
lower corner
$\in\partial_E\L$), $k_i\neq M$ since $u_M=v^*$.
\end{remark}
Using Lemma~\ref{svegliadalle3}, it is simple to check by induction
that the above algorithm is well
posed, it always stops and that exactly $S-1$ points among
$z_1,\ldots,z_S$ belong to $\L$.

It is convenient to parametrize the points $z_1,\ldots,z_S$
as follows. Let $\D_0:= \{z_0\}$, let $\e_1=-1$ and set $\e_i=\pm1$ if
$z_i\in\D_{i-1}^\pm$, $i=2,\ldots,S$.
If $\{v^*(\D_i),v_*(\D_i)\}$ denote the upper and lower corner,
respectively, of the box $\D_i$, then by construction,
$z_i \ll v_*(\D_{i-1})$ if $\e_i=-1$ and $v^*(\D_{i-1}) \ll z_i$
otherwise. Finally, we define
\[
\xi_i:= \cases{ v_*(\D_{i-1})- z_i, & \quad $\mbox{if } \e_i =-1 $,\vspace*{2pt}
\cr
z_i-v^*(
\D_{i-1}), &\quad $\mbox{if } \e_i=+1 $,}\qquad 1\leq i \leq S.
\]
Note that each $\xi_i$ has nonnegative coordinates and $\xi_i \neq0$.
By the previous considerations and by the definition of the sets
$\D_i^\pm$ (cf. Lemma~\ref{svegliadalle3}), if
$\gamma_i:= \|\xi_i\|_1$ and $\ell_i:= \|\D_i\|_1$ then
%
\begin{equation}\qquad
\label{relazione} \gamma_1=\ell_1=1 ,\qquad
\ell_{i+1}=\ell_i+ \gamma_{i+1} ,\qquad 1\leq
\gamma_{i+1} \leq\ell_i\ \forall i=1,\ldots,S-1.
\end{equation}
From the above identities, we get $\gamma_{i+1}\le\sum_{j=1}^i
\gamma_j$ and
$\ell_{i+1} \leq2 \ell_i$, that is, $\ell_i \leq2^{i-1}$.
On the other hand, when the algorithm stops for $i=S$, the box $\D_S$
has at least one edge of length $L$.
That implies that $2^n = L\leq\|\D_S\|_1=\ell_S\le2^{S-1}$, that is,
$S\geq n+1$.

\subsubsection{Counting the number of possible outputs} We now focus on
bounding from above the number $\cZ$ of the possible $(n{+}1)$-tuples
$(z_1, \ldots, z_{n+1})$ that can be produced by the above algorithm.
As already discussed, the vertices $(z_1, \ldots, z_{n+1})$ are uniquely
specified by $z_0$, by the vectors $(\xi_1,
\ldots, \xi_{n+1})$ and by the variables $(\e_1, \ldots,
\e_{n+1})$. Clearly, $z_0$ and $(\e_1, \ldots,
\e_{n+1})$ can be chosen in at most $L^d\times2^n=2^{(d+1)n}$ ways.
To upper bound the number $\Xi$ of the possible $(n+1)$-tuples $( \xi_1,
\ldots, \xi_{n+1})$,
we first observe that, given the lengths $(\gamma_1, \ldots,
\gamma_{n+1})$, there are at most $  [(1+\gamma_1)(1+
\gamma_2)\cdots(1+\gamma_{n+1}) ]^{d-1}$ possible
$(n+1)$-tuples $( \xi_1,
\ldots, \xi_{n+1})$.
Since $\gamma_{i+1}\le\sum_{j=1}^i \gamma_j$, setting
\[
U(k):= \bigl\{ (x_1, x_2, \ldots, x_k)\in
\bbN^k \dvtx x_1=1 \mbox{ and } 1 \leq x_i
\leq x_1+ \cdots+ x_{i-1}\ \forall i\dvtx 2\leq i \leq k
\bigr\}
\]
and writing $\sum_{U(k)}(\cdot)$ for the sum
restricted to values in $U(k)$, we get
\begin{eqnarray*}
\Xi&\le&\sum_{U(n+1)} \bigl[(1+x_1) (1+
x_2)\cdots(1+x_{n+1}) \bigr]^{d-1}
\\
&\le& \sum_{U(n+1)}2^{(d-1)(n+1)} [x_1
x_2\cdots x_{n+1} ]^{d-1}
\\
&\le&2^{(2d-1)(n+1)} \frac{ 2^{ d {n\choose2} } }{n! d^n},
\end{eqnarray*}
where we used Claim~\ref{lem:int-bnd} below.
In conclusion,
%
\begin{equation}
\label{cardinality} \cZ\le2^{(d+1)n}\Xi\le2^{3d(n+1)} \frac{ 2^{ d {n\choose2} } }{n! d^n}.
\end{equation}
%
\begin{claim}\label{lem:int-bnd} The following holds:
\[
\sum_{U(n+1)}(x_1x_2\cdots
x_{n+1})^{d-1}\leq\frac{ 2^{ d {n\choose2} +dn } }{n! d^n}\qquad \forall n>1.
\]
\end{claim}
\begin{pf}
Setting $M_n:=\sum_{i=1}^n x_i$ and summing over $x_{n+1}$
gives the bound
%
\begin{equation}
\label{maratona1} \sum_{U(n+1)}(x_1x_2
\cdots x_{n+1})^{d-1}\leq\sum_{U(n)}
(x_1x_2\cdots x_{n})^{d-1}
\frac
{(M_{n}+1)^d }{d},
\end{equation}
where we used the bound
\[
\sum_{i=1}^n f(i) \le\int
_0^{n+1}\,dx\, f(x),
\]
valid for any nonnegative increasing function $f$.
Similarly,
%
\begin{eqnarray}
\label{maratona2} a(j,k)&:=& \sum_{ U(j)}
(x_1x_2\cdots x_{j})^{d-1}(M_{j}+1)^k
\nonumber
\\
&\leq&\sum_{U(j-1)} (x_1x_2
\cdots x_{j-1})^{d-1} \int_0^{M_{j-1}+1 }
\,d x_j (x_j +M_{j-1}+1 )^{k+d-1}
\nonumber
\\[-8pt]
\\[-8pt]
\nonumber
&\leq&\frac{2^{k+d }}{k+d} \sum_{U(j-1)}
(x_1x_2\cdots x_{j-1})^{d-1}
(M_{j-1}+1) ^{k+d}
\\
&=& \frac{2^{k+d }}{k+d} a(j-1, k+d).\nonumber
\end{eqnarray}
If we combine together \eqref{maratona1} and \eqref{maratona2}, we obtain
\begin{eqnarray*}
&&d \sum_{U(n+1)} (x_1x_2
\cdots x_{n+1})^{d-1} \\
&&\qquad\leq a(n,d)
\\
&&\qquad\leq \frac{ 2^{2d} }{2d} a(n-1,2d)\leq\frac{ 2^{2d}\cdot2^{3d} }{(2d)(3d)
} a(n-2,3d)
\\
&&\qquad\leq\cdots\leq \frac{ 2^{2d}\cdot2^{3d} \cdots2^{nd} }{(2d)(3d)\cdots(nd) } a(1,nd)\leq\frac{ 2^{ d
{n\choose2} +d(n-1)} }{n! d^{n-1}}.
\end{eqnarray*}
\upqed\end{pf}
%
\subsection{Conclusion}
By the arguments above, we know that
\[
\partial A_* \subset \bigl\{\h\in\O_\L\dvtx \exists(z_1,
\ldots,z_{n+1})\in \G _\L^{(n)} \mbox{ with }
\h_{z_1}=\cdots=\h_{z_{n+1}} = 0 \bigr\} ,
\]
where $\G_\L^{(n)}$ consists of \emph{all} possible $(n+1)$-tuples
$(z_1,\ldots,z_{n+1})$ in $\bar\L$ which can be obtained by applying the
algorithm above to a pair $(\eta,z_0)$ satisfying $\h
\in\partial A_*$, $c^{\L, \mathrm{ max}
}_{z_0}(\eta)=1$ and $\eta^{z_0} \notin A_*$. Note that the set
$\G_\L^{(n)}$ has cardinality $\cZ$.
Thus, using~\eqref{cardinality} together with $n\leq\theta_q/d$,
\[
\pi(\partial A_*) \leq q^{n} \cZ\leq2^{-n\theta_q}2^{3d(n+1)}
\frac{ 2^{ d {n\choose2} } }{n! d^n} = \frac{ 2^{-n \theta_q+ d
{n\choose2}+O(\theta_q) } }{n! d^n}.
\]
By applying the trivial bound $\cD^\mathrm{ max}_\L(\1_{A_*}) \leq
L^d\pi(\partial A_*)$ (cf. Section~\ref{bottleneck}) we immediately get
\eqref{eq:dirbndbis}.
Finally, recall that $\mathds10\in A_*$ and $\mathds1 \notin A_*$. Thus,
\begin{eqnarray*}
\pi(A_*) &\geq&\pi(\mathds{1} 0 ) \geq q (1-q)^{L^d}\ge q/2,
\\
\pi \bigl(A^c_* \bigr)&\geq&\pi( \eta_{v^*}=1)=p \geq1/2,
\end{eqnarray*}
for $q$ sufficiently small (here the restriction $n\le\theta_q/d$ is crucial).
This completes the proof of the Theorem~\ref{pupi}.

\section{Proof of Theorem \texorpdfstring{\protect\ref{th:main1}}{1}}
\label{sec:5}

\subsection{Lower bound}

Recall that $\theta_q=\log_2(1/q)$ and let
$L_c= \lfloor2^{\theta_q/d}\rfloor$. Lem\-ma~\ref{lem:ineq} together
with \eqref{biancaneve} and Theorem~\ref{pupi} imply a more refined lower bound of the form
\[
\trel^\mathrm{ max}(L_c;q) \ge2^{{\theta_q^2}/{(2d)} + ({\theta_q}/{d})\log_2
\theta_q + O(\theta_q)}.
\]
%
Therefore, the $o(1)$ term in $\trel({\mathbb Z}^d;q)\ge2^{
{\theta
_q^2}/{(2d)}(1+o(1))}$ is $\O((\log_2 \theta_q)/\theta_q)$ (see Remark~\ref
{controllo_errori}).
Using \eqref{eq:East}, the RHS above can also be rewritten as $\trel
({\mathbb Z};\break q)^{({1}/d) (1+o(1))}$.


\subsection{Upper bound}\label{lagustelli}
We upper bound the relaxation time $\trel({\mathbb Z}^d;q)$ by a
renormalization procedure
based on the following result.
%
\begin{lemma}
\label{blok-versus-ss} Fixed $\ell\in\bbN$, set $q^*=1-(1-q)^{\ell
^d}$. Then for any $q \in(0,1)$
%
\begin{equation}
\label{eq:5} \trel \bigl({\mathbb Z}^d;q \bigr)\le
\kappa_d\trel^\mathrm{ min}(3\ell;q) \trel \bigl({\mathbb
Z}^d; q^* \bigr)
\end{equation}
for some constant $\kappa_d$ depending only on the dimension $d$.
\end{lemma}
We postpone the proof to the end of the section and explain how to
conclude. First, we note that,
since $q^*= \pi( \exists x \in\L_\ell\dvtx \eta_x=0)$, the Bonferroni
inequalities
(cf., e.g., \cite{durrett}) imply that
%
\begin{equation}
\label{eq:Bonfe} q \ell^d/2\le q^*\le q \ell^d\qquad\mbox{for
} q\ell^d\le1.
\end{equation}
We will now use Theorem~\ref{th:main2} together with \eqref{eq:5} to
prove inductively the required upper bound on $\trel({\mathbb Z}^d;q)$ as
$q\downarrow0$. Fix $d>1$.
We already know [cf. \eqref{eq:1}] that $\trel({\mathbb Z}^d;q)\le
\trel({\mathbb Z};q)$ so that, using \eqref{eq:East},
\[
\trel \bigl({\mathbb Z}^d;q \bigr)\le 2^{{\theta_q^2}/{2}+\theta_q\log_2 \theta_q + \gamma_0 \theta
_q+\a_0},
\]
for some constants $\gamma_0,\a_0>0$ and any $q\in(0,1)$. Assume now that,
for some $\lambda\in
(1/d, 1]$ and $\gamma,\a>0$, the following bound holds for all $q \in(0,1)$:
%
\begin{equation}
\label{eq:6} \trel \bigl({\mathbb Z}^d;q \bigr)
\le2^{\lambda({\theta_q^2}/{2})+\theta
_q\log_2
\theta_q +
\gamma\theta_q+\a}.
\end{equation}
Choose the free parameter $\ell$ in \eqref{eq:5} of the form $\ell
=2^n$ with
$1\le n\le\theta_q/d$. With this
choice and using \eqref{eq:Bonfe}, we get
\begin{equation}
\label{laurea} \theta_{q^*}\le\theta_q -nd + 1\le
\theta_q.
\end{equation}
Using \eqref{eq:5} and
\eqref{eq:6} together with Theorem~\ref{th:main2} to bound from above
the term $\trel^\mathrm{ min}(3\ell;q)$ for all $q \in(0,1)$ by
\[
\trel^\mathrm{ min}(3\ell;q) \le2^{n\theta_q - {n^2}/{2}
+n\log_2 n + \b\theta_q+\rho}
\]
for some constants $\b,\rho>0$ independent of $n$, we get
%
\begin{eqnarray}\label{arance}
\trel \bigl({\mathbb Z}^d;q \bigr)&\le&\kappa_d
2^{n\theta_q - {n^2}/{2}
+n\log_2 n + \beta\theta_q + ({\lambda}/2) \theta^2_{q^*}+\theta
_{q^*}\log_2 \theta_{q^*} +
\gamma\theta_{q^*}+\a+\rho}
\nonumber
\\[-8pt]
\\[-8pt]
\nonumber
&\le&\kappa_d 2^{n\theta_q - {n^2}/{2}
+ ({\lambda}/2) \theta^2_{q^*}+\theta_{q}\log_2 \theta_{q} +
(\gamma+\b) \theta_{q}+\alpha+\rho}. 
\end{eqnarray}
Above we used that
\[
n\log_2 n + \theta_{q^*}\log_2
\theta_{q^*} + \gamma\theta_{q^*} \leq(n+
\theta_{q^*}) \log_2 \theta_q + \gamma
\theta_{q^*}\leq \theta_q \log_2
\theta_q + \gamma\theta_q,
\]
where the first inequality follows from $\max\{ n,
\theta_{q^*}\}\le\theta_q$ and the latter from \eqref{laurea}.
Using again \eqref{laurea}, we can bound
%
\begin{eqnarray}
\label{limoni} n\theta_q - \frac{n^2}{2} + \frac{\lambda}2
\theta^2_{q^*} &\leq& n\theta_q -
\frac{n^2}{2} + \frac{\lambda}2 ( \theta_q -nd +
1)^2
\nonumber
\\
&=& \frac{n^2}{2} \bigl(d^2 \lambda-1 \bigr) -n \bigl(
\theta_q (\lambda d -1) +\lambda d \bigr) + \frac{\lambda}{2} (
\theta_q +1)^2 \\
&=:& \frac{n^2}{2}A- nB+C.\nonumber
\end{eqnarray}
Note that $A,B>0$.
We now optimize over $n$ and
choose it equal to $n_c=\lfloor B/A\rfloor$, that is,
\[
n_c= \biggl\lfloor\frac{ \theta_q (\lambda d -1)+\lambda d
}{d^2\lambda-1} \biggr\rfloor.
\]
Since $\frac{n_c^2}{2}A- n_cB+C \leq- \frac{B^2}{2A}+B+C$, 
from \eqref{arance} and \eqref{limoni} we derive that
\begin{eqnarray*}
\trel \bigl({\mathbb Z}^d;q \bigr)&\le&\kappa_d
2^{ - ({ [ \theta_q(\lambda d
-1)+\lambda d]^2
}/{(2(d^2\lambda-1))} )}\\
&&{}\times 2^{ \theta_q (\lambda d -1) +\lambda d +
({\lambda}/{2}) (\theta
_q+1)^2+ \theta_{q}\log_2 \theta_{q} +
(\gamma+\b) \theta_{q} +\a+\rho.
}
\end{eqnarray*}
Hence, using that $\lambda\in(1/d,1]$, we conclude that for any $q
\in(0,1)$
\[
\trel \bigl({\mathbb Z}^d;q \bigr) \le 2^{({\theta_q^2}/{2})\lambda_1 +\theta_{q}\log_2 \theta_{q} +
\gamma_1\theta_{q}+\a_1},
\]
where $\lambda_1= \frac{2d\lambda-1-\lambda}{d^2\lambda-1}$,
$\gamma_1=\gamma+\b+ d$ and $\a_1= \a
+\rho+ d + 1 + \log_2\kappa_d $.

We interpret the above as a three-dimensional dynamical system in the
running coefficients $(\lambda,\gamma, \a)$. Let $(\lambda_k,\gamma
_k, \a_k)$ be the
constants obtained after $k$ iterations
of the above mapping starting from $\lambda_0=1$, $\gamma_0$, $\a
_0$. Clearly,
$\gamma_k,\a_k=O(k)$. As far as $\lambda_k$ is concerned, it is easy
to check
that the sequence is decreasing
under recursive application of the map
\[
(1/d, 1] \ni\lambda\mapsto\frac{2d\lambda-1-\lambda}{d^2\lambda
-1} \in(1/d, 1]
\]
and it has an
attractive quadratic fixed point at $\lambda_c=1/d$. Thus, $\lambda_k=
\lambda_c+O(k^{-1})$.
Choosing $k=\lfloor\theta_q^{1/2}\rfloor$, we then
we get (in agreement with Remark~\ref{controllo_errori})
\[
\trel \bigl({\mathbb Z}^d;q \bigr) \le2^{\lambda_k({\theta_q^2}/{2}) +
\theta
_q\log_2 \theta_q
+\gamma_k \theta_q+\a_k}
=2^{{\theta_q^2}/{(2d)}+O(\theta_q^{3/2})}.
\]
%

\begin{pf*}{Proof of Lemma~\ref{blok-versus-ss}}
Consider the East-like block process defined in
Section~\ref{block process} (cf. Definition~\ref{Eblock}).
Due to Proposition~\ref{prop-block}, it is enough to prove for any
$\ell\in\bbN$ and $q \in(0,1)$ that
%
\begin{equation}
\label{eq:4} \trel \bigl({\mathbb Z}^d;q \bigr)\le
\kappa_d \trel^\mathrm{ min}(3\ell;q) \trel (\cL
_\mathrm{ block}).
\end{equation}
In order to prove the above bound, we need to define another auxiliary chain.
%
\begin{definition}[(The Knight chain)]
\label{Knight Chain} On the vertex set $V:={\mathbb Z}^d$ define the
following graph
structure $G=(V,E)$. Given two vertices $x=(x_1,\ldots,x_d)$ and
$y=(y_1,\ldots,y_d)$
we write $y\prec x$ if there exists $j\in\{1,\ldots,d\}$ such that
$y_i=x_i-1, \forall i\neq j$ and $y_j=x_j-2$. Then we define
the edge set $E$ as those pairs of vertices $(x,y)$ such that either
$y\prec x$ or $x\prec y$.
It is easy to see that $G$ is the union of $d+1$ disjoint subgraphs
$G^{(i)}=(V^{(i)},E^{(i)})$, each one isomorphic to the original
lattice ${\mathbb Z}^d$ (cf. Figure~\ref{fig:test}).

The \emph{Knight chain} $KC(\ell)$ with parameter $\ell\in\bbN$ is
then defined very similarly to the
East-like block process (cf. Definition~\ref{Eblock}) except that the
constraint is tailored to the graph $G$.
Partition ${\mathbb Z}^d$ into blocks of the form
$\L_\ell(x):=\L_\ell+ \ell x$, $x\in{\mathbb Z}^d$, where $\L
_\ell=[1,
\ell]^d$.
On $\O=\{0,1\}^{{\mathbb Z}^d}$ define
the Markov process which, with rate one and independently among the
blocks $\L_\ell(x)$, resamples from $\pi_{\L_\ell(x)}$ the
configuration in the
block $\L_\ell(x)$ provided that the Knight-constraint $c_x^{(\mathrm{kc})}$ is
satisfied, where $c_x^{(\mathrm{ kc})}$ is the indicator of the event that
for some $y\prec x$, the
current configuration in the block
$\L_\ell(y)$ contains a vacancy.
\end{definition}
%
%
\begin{figure}
\centering
\begin{tabular}{@{}cc@{}}

\includegraphics{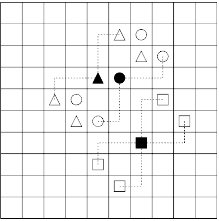}
 & \includegraphics{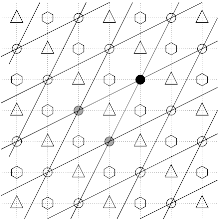}\\
\footnotesize{(a)} & \footnotesize{(b)}
\end{tabular}
\caption{\textup{(a)} The blocks forming the underlying grid
are unit squares centered at the vertices of the lattice
${\mathbb Z}^2$. The four blocks containing the white squares
are the
neighbors of the block with the black square; similarly, for the
blocks containing the white circles and the white
triangles. \textup{(b)} A larger image showing the vertices of
$\{G^{(i)}\}_{i=1}^3$ and the edges of $G^{(1)}$.}
\label{fig:test}
\end{figure}

Because of the structure of the graph $G$, the chain $KC(\ell)$ is a
product chain, one for each subgraph $G^{(i)}$, $i=1,\ldots,d+1$, in
which each factor is
isomorphic to the East-like block process. Hence, its relaxation time
$\trel(KC(\ell))$ coincides with that of the East-like block process
$\trel(\cL_\mathrm{ block})$.
We can therefore write the Poincar\'e inequality
\[
\var(f)\le\trel(\cL_\mathrm{ block})\sum_{i=1}^{d+1}
\sum_{x\in
V^{(i)}}\pi \bigl(c^{(\mathrm{ kc})}_x
\var_{\L_\ell(x)}(f) \bigr)\qquad \forall f\in L^2(\pi).
\]
Using the enlargement trick (cf. Lemma~\ref{lem:enlarge}) together with
Lemma~\ref{lem:ineq}, we get that
\[
\pi \bigl(c^{(\mathrm{ kc})}_x\var_{\L_\ell(x)}(f) \bigr)\le
\trel^\mathrm{ min}(3\ell; q )\sum_{z\in\L_{3\ell}+\ell x'}\pi
\bigl( c_z \var_z(f) \bigr),
\]
where $x'_i=x_i-2$ for all $i=1,\ldots, d$. Therefore,
\[
\var(f)\le\kappa_d\trel^\mathrm{ min}(3\ell; q )\trel(
\cL_\mathrm{ block}) \sum_{x\in{\mathbb Z}^d}\pi
\bigl(c_x\var_x(f) \bigr),
\]
for some constant $\kappa_d$ depending only on the dimension $d$. By
definition, the latter implies \eqref{eq:4}.
\end{pf*}

%
%

\section{Proof of Theorem \texorpdfstring{\protect\ref{th:main2}}{2}}
\label{sec:6}
Without loss of generality, due to Lemma~\ref{lem:ineq} and since
$n(q) \to\infty$,
in the proof of \eqref{eq:2} and \eqref{eq:3}
we fix the side $L$ of
$\L$ equal to $2^n$.

\subsection{Maximal boundary conditions}

\subsubsection{Upper bound in (\texorpdfstring{\protect\ref{eq:2}}{2.8})}
If $n\ge\theta_q/d$, we can use Lemma~\ref{lem:ineq}
together with Theorem~\ref{th:main1} to get
\[
\trel^\mathrm{ max}(L;q)\le \trel \bigl({\mathbb Z}^d;q
\bigr)=2^{({\theta_q^2}/{(2d)})(1+o(1))}.
\]
For $n\le\theta_q/d$, we proceed as in the proof of the upper
bound in Theorem~\ref{th:main1}. Without loss of generality, we can
assume $d\ge2$ since the result was proved in \cite{CFM}, Theorem~2,
for $d=1$.

Fix $\ell=2^m$ with $m<n$, let $J\equiv J_{\ell,L}= [0,L/\ell-1]^d$
and for $x\in J_{\ell,L}$ let $\L_\ell(x)=[1,\ell]^d+\ell x$.
Then we have the analog of Lemma~\ref{blok-versus-ss}.
%
\begin{lemma}
\label{blok-versus-ss-bis}
Setting
$q^*=1-(1-q)^{\ell^d}$, we have
%
\begin{equation}
\label{eq:tosse} \trel^\mathrm{ max}(L;q)\le\kappa_d
\trel^\mathrm{ min}(3\ell;q) \trel^\mathrm{ max} \bigl(L/\ell;q^*
\bigr)
\end{equation}
for some constant $\kappa_d$ depending only on the dimension $d$.
\end{lemma}
\begin{pf}
We sketch the proof, which is essentially the same as the proof of Lemma~\ref{blok-versus-ss} apart boundary effects. Recall the notation
introduced in Definition~\ref{Knight Chain} (in particular, the partial
order $y\prec x$) and define the finite-volume Knight chain on $\O_\L$
as the Markov chain with generator
\[
\cL_\mathrm{ KC,J} f(\eta):= \sum_{x \in J}
\hat c_x (\eta) \bigl[ \pi _{\L
_\ell(x) } (f) -f \bigr](\eta) ,
\]
where $\hat c_x (\eta)$ is the characteristic function that there
exists $y \in{\mathbb Z}^d$ with $y \prec x$ such that $\eta$ has a vacancy
in $\L_{\ell} (y)$ (we extend $\eta$ as zero
outside $\L$).
Using the enlargement trick (cf. Lemma~\ref{lem:enlarge}) as
in the proof of Lemma~\ref{blok-versus-ss}, we get
%
\begin{equation}
\label{agretti} \var(f)\le\kappa_d\trel^\mathrm{ min}(3\ell;
q ) \trel(\cL_\mathrm{ KC,J} ) \sum_{z\in\L}\pi
\bigl(c^{\L,\mathrm{ max} }_z\var_z(f) \bigr),
\end{equation}
for some constant $\kappa_d$ depending only on the dimension $d$.
Above $\trel(\cL_\mathrm{ KC,J} )$ denotes the relaxation time of the
finite--volume Knight chain. Since this chain
is a product of
$d+1$ independent Markov chains, each one with generator
\[
\cL^{(i)}_\mathrm{ KC,J} f(\eta):= \sum
_{x \in V^{(i)}\cap J} \hat c_x (\eta) \bigl[ \pi_{\L_\ell(x) }
(f) -f \bigr](\eta) ,
\]
it follows that $\gap ( \cL_\mathrm{ KC,J} ) = \min \{
\gap
 (\cL^{(i)}_\mathrm{ KC,J} ) \dvtx 1\leq i \leq d+1 \}$. On
the other hand, by Proposition~\ref{prop-block},\hskip.2pt\footnote{Although the
proposition is stated for ${\mathbb Z}^d$, the same proof works in the
present setting.} $\gap (\cL^{(i)}_\mathrm{ KC,J} )= \gap
 ( \cL^{(i)}; q^*)$ where $\cL^{(i)}$
is the generator of the East-like process on $V^{(i) }\cap J$,
thought of as subgraph of $G^{(i)}= (V^{(i)}, E^{(i)})$, with maximal
boundary condition
\[
\cL^{(i)} f(\s) = \sum_{x \in V^{(i)}\cap J }
c^{(i)}_x (\s) \bigl[ \pi_x (f) - f \bigr] (\s)
, \qquad\s\in\{0,1\} ^{V^{(i) }\cap
J} ,
\]
$c^{(i)}_x (\s)$
being the characteristic function that $\s$ has a vacancy at some $y
\prec x$, $y \in V^{(i)}$ (set $\s\equiv0$ on $V^{(i)}\setminus J$).

We now observe that the set $V^{(i)}\cap J$, endowed with the graph
structure induced by $G^{(i)}$, is isomorphic to a subset $A^{(i)}$ of
$[1,L/\ell]^d$ (see Figure~\ref{fig:Grecia}).

%
\begin{figure}

\includegraphics{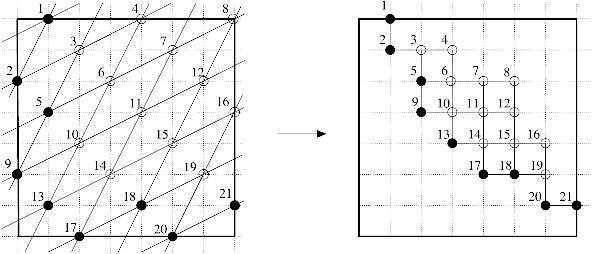}

\caption{Left: The square represents $J=[0, L/\ell-1]^d$ with $d=2$,
$L/\ell= 2^3$. Circles
mark points in $J \cap V^{(i)}$, where the index $i$ is such that
$(1,1) \in V^{(i)}$, and transversal lines give the edges induced by
$G^{(i)}$. Black circles mark points $x \in J \cap V^{(i)}$ with $c_x
^{(i)} \equiv1$ (constraint always fulfilled). Right: Circles mark
points in $A^{(i)}$, the black ones correspond to points with fulfilled
constraint. The isomorphism maps marked points on the left to marked
points on the right maintaining the enumeration.}
\label{fig:Grecia}
\end{figure}

Using this isomorphism,
the process generated by $\cL^{(i)}$ can be identified with the
East-like process on
$\O_{A^{(i)}}$ with maximal boundary conditions and parameter $q^*$.
By the same arguments leading to \eqref{eq:degrado} in Lemma~\ref{lem:ineq}, we get that $\trel(\cL^{(i)};q^*)\le\trel^\mathrm{
max}(L/\ell;q^*)$ and, therefore, the same upper bound holds for
$\trel
(\cL_\mathrm{ KC,J} )$.
\end{pf}
By \eqref{eq:3} we know that, for some positive constants $\a,\bar\a$
and for all integer $r \leq\theta_q$ it holds
%
\begin{equation}
\label{cane_bello} \trel^\mathrm{ max} \bigl(2^r;q \bigr) \leq
\trel^\mathrm{ min} \bigl(3 \cdot2^r;q \bigr)\leq
2^{r\theta_q - {r^2}/{2}+r\log_2 r + \a\theta_q+ \bar\a}.
\end{equation}
Given positive constants
$\lambda, \b, \bar\b$, we say that property $P(\lambda, \b, \bar
\b)$ is
satisfied if
\begin{equation}\quad
\label{gattino} \trel^\mathrm{ max} \bigl(2^n;q \bigr) \leq
2^{n\theta_q - \lambda({n^2}/{2})+n\log_2 n +\b\theta_q+ \bar
\b} \qquad \forall q \in(0,1),\forall n \leq\theta_q/d.
\end{equation}
Note that, due to \eqref{cane_bello}, property $P(1, \a, \bar\a)$ is
satisfied. The following result is at the basis of the renormalization
procedure.
%
\begin{lemma}\label{nuvoloni} If property $P(\lambda, \b, \bar\b)$ is
satisfied with $\lambda\leq d$, then also property $P(\lambda', \b',
\bar\b')$
is satisfied, where
$ \lambda' := \frac{d^2-\lambda}{2d-\lambda-1}\leq d$,
$ \b':=\a+ \b+1$ and
$ \bar\b':=\bar\a+ \bar\b+ d$.
\end{lemma}
The proof follows from Lemma~\ref{blok-versus-ss-bis} and
\eqref{cane_bello} by straightforward computations similar to the ones
of Section~\ref{lagustelli} and we omit it here. The interested reader
can find all the details
in Appendix B of the extended version.\footnote
{\url{http://arxiv.org/abs/1404.7257}.}

Let $H(\lambda)= \frac{d^2-\lambda}{2d-\lambda-1}$. We interpret
the map $(\lambda, \b
,\bar\b)\mapsto(\lambda', \b', \bar\b')$ in Lem\-ma~\ref
{nuvoloni} as a
dynamical system. Let $(\lambda_k,\b_k, \bar\b_k)$ be the constants
obtained after $k$ iterations
of the above mapping starting from $( 1, \a, \bar\a)$. Clearly,
$\b_k,\bar\b_k=O(k)$, while the map $H$ has an attractive quadratic
fixed point at $d$, thus implying that $\lambda_k=
d+O(k^{-1})$. Restricting to $q \in(0,1/2]$, we then obtain that
%
\begin{eqnarray}
\label{leone_bis} \trel^\mathrm{ max} \bigl(2^n;q \bigr) \leq
2^{n\theta_q - d({n^2}/{2})+n\log_2 n +({c}/{k}) ({n^2}/{2})
+ ck
\theta_q }
\nonumber
\\[-8pt]
\\[-8pt]
\eqntext{\forall q \in(0,1/2], \forall n \leq\theta_q/d
,}
\end{eqnarray}
for a constant $c>0$ depending only on $d$, thus implying the thesis.

\begin{remark}\label{patria}
One can optimize \eqref{leone_bis} by taking $k := \lceil\sqrt{ n^2
/2 \theta_q } \rceil$. As a result, one gets a better upper bound
w.r.t. \eqref{eq:2} when $n \gg\theta_q^{1/2}$. More precisely, one gets
%
\begin{eqnarray}
\label{leone_tris} \trel^\mathrm{ max} \bigl(2^n;q \bigr) \leq
2^{n\theta_q - d({n^2}/{2})+n\log_2 n +c n \theta_q ^{1/2} }
\nonumber
\\[-8pt]
\\[-8pt]
\eqntext{\forall q \in(0,1/2], \forall n \in\bigl( c'
\theta_q^{1/2}, \theta_q/d\bigr] ,}
\end{eqnarray}
for suitable constants $c,c'>0$ independent from $n,q$.
\end{remark}
%


\subsubsection{Lower bound in (\texorpdfstring{\protect\ref{eq:2}}{2.8})}
Using Lemma~\ref{lem:ineq}, it is enough
to prove the lower bound for $L=2^n$ with $n\le\theta_q/d$. In this
case, the sought lower bound follows from the bottleneck inequality
\eqref{biancaneve} together with Theorem~\ref{pupi}. More precisely,
one gets
%
\begin{equation}
\label{tricolore} \trel^\mathrm{ max} \bigl(2^n;q \bigr) \geq
2^{n\theta_q -d
{n\choose2}+n\log_2 n +O(\theta_q) }, \qquad n\le\theta_q/d.
\end{equation}

\subsection{Minimal boundary conditions}
\subsubsection{Lower bound in \texorpdfstring{\protect\eqref{eq:3}}{(2.9)}}
By Lemma~\ref{lem:ineq}, $\trel^\mathrm{ min}(L;q)$
is bounded from below by the relaxation time of the East process on the
finite interval $[1, L]$ with parameter $q$. If $n\le\theta_q$, the
required lower bound of the form of the RHS of \eqref{eq:3}
then follows from \cite{CFM}, Theorem~2. If instead $n\ge\theta_q$, we
can use
the monotonicity in $L$ of $T^\mathrm{ min}_\mathrm{ rel} (L;q)$ (cf. Lemma~\ref{lem:ineq}) to get
$\trel^\mathrm{ min}(L;q)\ge\trel^\mathrm{ min}(2^{\theta_q};q)$.

\subsubsection{Upper bound in \texorpdfstring{\protect\eqref{eq:3}}{(2.9)}}
\label{sec:tree}
Given $\L=[1,L]^d$ consider the rooted directed graph $G=(V,E,r)$ with
vertex set
$V=\L$, root $r=(1,\ldots,1)$ and edge set $E$ consisting of all pairs
$(x,y)\in V\times V$ such that $y=x +e$ for some $e\in\cB$. Notice
that for any
$v\in V$ there is a path in $G$ from $r$ to $v$. Using this property,
it is well known that the graph $G$ contains a \emph{directed spanning
tree} (or
\emph{arborescence}) rooted at $r$, that is, a subgraph $\cT=(V,F)$ such
that the underlying undirected graph of $\cT$ is a spanning tree
rooted at $r$ of
the underlying undirected graph of $G$ and for every $v\in V$ there is
a path in $\cT$ from $r$ to $v$ (cf., e.g., \cite{Gibbons}). In the
present case, it is simple to build such a $\cT$.

Let $\cT$ be one such directed spanning tree and let us consider a modified
East-like process on $\L$ with the new constraints:
\[
c_x^{\cT,\mathrm{ min}}(\eta):= \cases{ 1, &\quad$\mbox{if either $x=r$ or $
\eta_y=0$ where $y$ is the parent of $x$ in $\cT$},$\vspace*{2pt}
\cr
0, &\quad $\mbox{otherwise.}$ }
\]
Clearly, $c_x^{\cT,\mathrm{ min}}\le c_x^{\L,\mathrm{ min}}$ so that
$\trel^\mathrm{ min}(L;q)\le\trel^\mathrm{ min}(\cT;q)$, where $\trel^\mathrm{
min}(\cT;q)$ denotes the relaxation time of the modified process. In
turn, as shown in
\cite{Praga}, Theorem~6.1 and equation (6.3), page~307, $\trel^\mathrm{
min}(\cT;q)$ is smaller than the relaxation time of the
one-dimensional East process on the longest branch of $\cT$, which has
$dL -d+1 $ vertices. Such a relaxation time was estimated quite
precisely in \cite{CFM}, Theorem~2, to be
equal to $2^{n\theta_q - {n\choose2} +n\log_2 n +O(\theta_q)}$
for $n\le\theta_q$ and to $2^{{\theta^2_q}/{2} +\theta_q\log_2
\theta_q +O(\theta_q)} $ for $n\ge\theta_q$ (cf. the discussion
before Theorem~\ref{th:main1}). This proves~\eqref{eq:3}. 

\subsubsection{Lower bound in \texorpdfstring{\protect\eqref{eq:333}}{(2.10)}}
We first need a combinatorial lemma which extends previous results for
the East process \cite{CDG}.
Consider ${\mathbb Z}_+^d$ and recall that $x_*=(1,1,\ldots, 1)$. Given
$\eta\in\O_{{\mathbb Z}_+^d}$, we write $|\h|:=|\{x\in{\mathbb
Z}^d_+ \dvtx \h_x =
0\}|$ for the total number of vacancies of $\h$. Moreover, we let $Z_m
:=\{ \eta\in\O_ {{\mathbb Z}^d_+} \dvtx |\eta|\le m\}$ and define $V_m $
as the set of configurations which, starting from the configuration $\1
$ on ${\mathbb Z}_+^d$ with no vacancy, can be reached
by East-like paths in $Z_m$ (i.e., paths for which each transition is
admissible for the East-like process in ${\mathbb Z}_+^d$ with minimal boundary
conditions, i.e., with a single frozen vacancy at $x_*-e$
for some $e
\in\cB$ and using no more than $m$ simultaneous other vacancies).
%
\begin{lemma}
\label{lemma:CDG} For $m \in\bbN$
\begin{eqnarray*}
Y(m) &:=& \max \bigl\{\|x-x_*\|_1+1 \dvtx x\in{\mathbb
Z}_+^d, \h_x = 0 \mbox{ for some } \h\in V_m
\bigr\} = 2^m-1 ,
\\
X(m)&:=& \max \bigl\{\|x-x_*\|_1+1 \dvtx x\in{\mathbb
Z}^d_+, \h_x = 0 \mbox{ for some } \h\in V_m ,
|\h| = 1 \bigr\} = 2^{m - 1}.
\end{eqnarray*}
\end{lemma}
%
\begin{remark}
Note that $\|x-x_*\|_1+1$ equals the $L^1$-distance between $x$ and
the frozen vacancy.
\end{remark}
\begin{pf*}{Proof of Lemma \ref{lemma:CDG}}
It is convenient to write $V_m(d)$, $Z_m(d)$ instead of
$V_m$, $Z_m$ in order to stress the $d$-dependence.
The lower bounds $X(m) \geq2^{m-1}$ and $Y(m) \geq
2^{m}-1$ follow immediately from the same result for the East
process (cf.~\cite{CDG}, Section~2) if we use that,
under minimal boundary conditions, the projection process to the
line $(x,1,1,\ldots,1), x\in{\mathbb Z}_+$ coincides with the East process
on ${\mathbb Z}_+$.

We now prove the upper bounds $X(m)\leq2^{m-1}$ and $Y(m) \leq
2^{m}-1$. To this aim given $a \in{\mathbb Z}_+$ we define $\G_a := \{
x \in
{\mathbb Z}_+^d \dvtx \|x-x_*\|+1=a \}$ [e.g., $\G_a= \{ (1,a), (2,a-2) ,
\ldots,
(a,1) \}$ for $d=2$]. We then
define the map $\rho: \O_{{\mathbb Z}_+^d} \mapsto\O_{{\mathbb
Z}_+}$ as
\[
\rho(\eta)_a:=\cases{ 1, &\quad $\mbox{if } \eta_x =1\
\forall x \in\G _a ,$ \vspace*{2pt}
\cr
0, &\quad $\mbox{otherwise}.$}
\]
Note that $\rho$ does not increase the number of vacancies. Moreover,
if $\gamma:=( \eta^{(1)}, \ldots, \eta^{(n)})$ is an East-like path in
$Z_m(d)$, then its image under $\rho$ is an East-like path in
$Z_m(1)$ (possibly with constant pieces). In particular, given an
East-like path in $Z_m(d)$ starting from $\1$, its $\rho$-image
gives an East-like path in $Z_m(1)$ starting from the full
configuration. Hence, $\rho( V_m(d) ) \subset V_m(1)$. Since the
thesis of the lemma is true for $d=1$ due to \cite{CDG}, Section~2, we
then recover that the maximal $a \in{\mathbb Z}_+$ such that a vacancy can
be created in $\G_a$ by some path $\gamma$ is bounded by $2^m-1$. On the
other hand, such a value $a$ equals $Y(m)$. Similarly, if
$ \eta\in V_m(d)$ has a single vacancy, then $\rho(\eta) \in V_m(1)$
has a single vacancy and the thesis for $d=1$ implies that
$X(m) \leq2^{m-1}$.
\end{pf*}

The previous combinatorial result allows us to construct a small
bottleneck which gives rise to the lower bound
in \eqref{eq:333} of Theorem~\ref{th:main2}. This bottleneck is of
energetic nature as in \cite{CMST}, the \hyperref[app]{Appendix} and \cite{CFM}, Lemma~5.5.

Take $\L= [1,L]^d$ with $\ell:=\|\L\|_1 + 1 = \|v^*-x_*\|_1+1 \in
(2^{n-1},2^n]$, let $\10\in\O_{\L}$ be the configuration with a single
vacancy located at the upper corner $v^*$.
Let $V=V_n$ be the set of configurations in $\O_\L$ which can be
reached from $\1$ by East-like paths (with minimal boundary conditions)
such that at each step there are at most $n$ vacancies in $\L$.
Clearly, $V \subseteq\{ \h_{\L} \dvtx \h\in V_n\}$.
Since $X(n) = 2^{n-1}$ and $\|v^*-x_*\|_1+1 > 2^{n-1}$, we have that $\1
0 \notin V$.
Also by definition $\1 \in V$, so $\pi(V) \geq\pi(\1) = 1 + o(1)$ and
$\pi(V^c) \geq\pi(\10) \geq q(1+o(1))$.

We now give a lower bound on $\cD^\mathrm{ min}_{\L} (\1_V)$.
Let $U := \{ \h\in\O_\L\dvtx |\h| = n\}$.
By definition, if $\h\in V$ then $|\h| \leq n$.
If $\h\in V$ and $|\h| < n$, then $\h^x \in V$ for each $x \in\L$
with $c_x^{\L,\mathrm{ min}}(\h) =1$,
therefore, $\partial V \subseteq U$.
Recall \eqref{rane}, and observe that to escape the set $V$ a vacancy
must be created, so
\begin{eqnarray*}
\cD^\mathrm{ min}_{\L} (\1_V) &=& \sum
_{\h\in\partial V}\pi_\L(\h )\cK^\mathrm{ min} \bigl(
\h,V^c \bigr) \leq\sum_{\h\in U}
\pi_\L(\h) \mathop{\sum_{x\in\L\dvtx \h_x=1}}_{ c_x^\mathrm{ min}(\h) = 1}q
\\
&\leq&\pi(U)d (n+1)q \leq d(n+1)c_0(n,d)q^{n+1} ,
\end{eqnarray*}
where $c_0(n,d)$ in the number of configurations in $[1,2^n]^d$ with
exactly $n$ vacancies.
The lower bound in \eqref{eq:333} follows from the bottleneck
inequality \eqref{biancaneve} applied with minimal boundary conditions,
the above estimate and the above lower bounds on $\pi(V)$ and $\pi(V^c)$.

\subsubsection{Upper bound in \texorpdfstring{\protect\eqref{eq:333}}{(2.10)}}
The upper bound of the relaxation time on $\L=[1,L]^d$
with $\|\L\|_1 +1 \in(2^{n-1},2^n]$ can be derived as for the upper
bound in \eqref{eq:3} above.
Consider the rooted directed graph $G=(V,E,r)$ with
vertex set
$V=\L$, root $r=(1,\ldots,1)$ and edge set $E$ consisting of all pairs
$(x,y)\in V\times V$ such that $y=x +e$ for some $e\in\cB$.
By the same argument as previously, $G$ contains a directed spanning
tree, and the longest branch contains exactly $\ell:=\|v^*-x_*\|_1+1=
\|
\L\|_1+1$ vertices.
It follows that the relaxation time is bounded above by the relaxation
time of the East process on $[1,\ell]$
which is known to be bounded above by $c(n)/q^n$ (see, e.g., (2.6) in
\cite{CFM}).

\section{Proof of Theorem \texorpdfstring{\protect\ref{th:main3}}{3}}
\label{sec:7}
\subsection{Proof of \texorpdfstring{\protect\eqref{hitmin}}{(2.11)}} Since the boundary
conditions are
minimal, the mean hitting time $T^\mathrm{ min}(v_*;q)$ coincides with the
same quantity in one dimension and for the latter \eqref{hitmin}
follows from \cite{CFM}, Theorems 1 and 2.
\subsection{Proof of \texorpdfstring{\protect\eqref{hitmax}}{(2.12)}}
\label{s.oreste}In agreement with Remark~\ref{pasqua}, we prove \eqref{hitmax} for a generic ergodic boundary
condition $\s$. Below $\L=[1,L]^d$.
\subsubsection{Lower bound}
Let $\tilde\t_{v^*}$ be the
hitting time of the set $\{\eta \dvtx \eta_{v^*}=1\}$. As in~\cite{CFM}, Proposition~3.2, the hitting time $\tilde\t_{v^*}$
starting from the configuration ${\mathds1}0$ with a single vacancy at $v^*$
is stochastically dominated by the hitting time $\t_{v^*}$ starting
with no
vacancies.
Thus, $T^{\s}(v^*;q)\ge\bbE^{\L, \s}_{{\mathds1}0}(\tilde\t
_{v^*}) $.
To lower bound the latter, we use the observation that the hitting time
$\tilde\t_{v^*} $ for the East-like process in ${\mathbb Z}^d_+$ coincides
with the same hitting time for the process in $\L=[1,L]^d$ together
with Lemma~\ref{lem:capbnds}.
Using the variational characterization \eqref{eq:dir-princ} of the
capacity together with the fact that the indicator ${\mathds1}_{A_*}$
of the bottleneck
$A_*$ constructed in Theorem~\ref{pupi} is zero on $\{\eta\in\O_\L
\dvtx
\eta_{v^*}=1\}$ and one on the configuration ${\mathds1}0$, we get that
\[
\bbE^{\L, \s}_{{\mathds1}0}(\tilde\t_{v^*}) \ge c
\frac{q}{\cD
^{\s
}_{\L}({\mathds
1}_{A_*})}\ge c \frac{q}{\cD^\mathrm{ max}_{\L}({\mathds1}_{A_*})}.
\]
The
sought lower bound follows at once from Theorem~\ref{pupi}.

\subsubsection{Upper bound}
Lemma~\ref{lem:capbnds}
and Remark~\ref{rem:capmono} imply that
%
\begin{equation}
\label{eq:TR} T^{\s} \bigl(v^*;q \bigr) \leq R^{\s}_{\1,B}
\leq R^\mathrm{ min}_{\1,B} , 
\end{equation}
where $\1$ denotes the configuration with no vacancies, $B=\{\eta\in
\O
_\L\dvtx\eta_{v^*}=0\}$. 
Thanks to Thompson's principle [see \eqref{eq:thom}]
the main idea now is to construct a suitable unit flow
and to bound its energy by a multiscale analysis.
In order to proceed, we need to fix some additional notation.

Given
$x=(x_1,\ldots,x_d)\in\L$, let $\L_x = \prod_{i=1}^d [1,x_i]$ and
$B_x = \{\eta\in\O_{\L_x} \dvtx \eta_x = 0\}$.
Next, we define
%
\begin{equation}
R(x) := R^{\L_x,\mathrm{ min}}_{\1,B_x} = \inf \bigl\{\cE(\theta) \mid\theta
\mbox{ a unit flow from $\1$ to $B_x$ in $\O_{\L_x} $} \bigr
\}.
\end{equation}
%

\begin{lemma}
\label{claim:main} Let $\L=[1,L]^d$ with $L=2^n$ and $n\leq\theta_q/d$.
Given $x\in\L$ with entries $x_i \geq3$, let $V_x$ be a box inside
$\prod_{i=1}^d[2,x_i-1]$ containing at least one lattice site and let
$\rho:V_x\mapsto[0,1]$ be such that $\sum_{y\in V_x}\r(y) =1$. Then
%
\begin{equation}\quad
R(x) \leq9 \sum_{y\in V_x} \r(y) R(y) +
\frac{9}{q} \sum_{y\in
V_x}\r ^2(y) R(y
)+ \frac{9}{q} \sum_{y\in V_x}\r^2(y)
R \bigl(\tilde x(y) \bigr),
\end{equation}
where $\tilde x(y) = x-y+(0,1,1, \ldots, 1)$.
\end{lemma}
Assuming the lemma, we complete the proof of the upper bound.
Given $N\in\bbN$, let
$L^\pm_m$ be defined
recursively by
\begin{eqnarray*}
L^+_{m}&=&2 L^+_{m-1} -\frac{1}{N}2^{m-1}
- 2,\qquad L_0^+=11,
\\
L^-_{m}&=&2 L^-_{m-1} +\frac{1}{N}2^{m-1}
+ 2,\qquad L_0^-=1.
\end{eqnarray*}
A simple computation gives
\[
L^+_m=2 +2^m \biggl(9 - \frac{ m }{2N} \biggr) ,\qquad
L^-_m=-2 +2^m \biggl(3 + \frac{ m }{2N} \biggr).
\]
%
It is straightforward to verify that the following occurs for $ 1\leq
m\le N$:
%
\begin{longlist}[(ii)]
\item[(i)]$L_m^-\le L_m^+$;
\item[(ii)] For any $x,y\in{\mathbb Z}^d$ such that $L_m^-\le x_i\le L_m^+$ and
$|2y_i-x_i|\le\frac{1}{N}2^{m-1}$ we have that both $y_i$ and
$x_i-y_i$ belong to the interval $[L_{m-1}^-+1,L_{m-1}^+-1]$.
\end{longlist}
%
\begin{lemma}
\label{lem:res}
Setting $R_m:= \max_{x\in[L_m^-,L_m^+]^d}R(x)$, 
\[
R_m\le27 \frac{N^d}{q2^{dm}}R_{m-1}, \qquad m_0 < m
\le N ,
\]
where $m_0 = \lceil\log_2 (4N  )\rceil$. In particular,
%
\begin{equation}
\label{eq:RN} R_N\le27^{N-m_0} 2^{(N-m_0)\theta_q - d  [ {N\choose2} - {m_0
\choose2}  ]+ d (N-m_0)\log_2 N }
R_{m_0}.
\end{equation}
\end{lemma}
\begin{pf}
Fix $x \in[L_m^-,L_m^+]^d$, and let $V_x = \{y\in\L_x \dvtx |2y_i - x_i|
\leq\frac{2^{m-1}}{N} \mbox{ for } 1\leq i \leq d\}$.
Observe that $|V_x| \geq (\frac{2^{m-1}}{N} -1  )^d \geq
(\frac{1}{N}2^{m-2}  )^d\geq1$, where in the second inequality we
have used $m > m_0=\lceil\log_2 (4N  )\rceil$.
Since $L_m^- \geq2^m> 2^2$, we have $x_i \geq4$, while $y_i, x_i -y_i
\geq L_{m-1}^-+1 \geq2$
[by (ii) above]. In particular, both $x$ and $V_x$ fulfill the
assumptions of Lemma~\ref{claim:main}.

By (ii) above, we have $y, \tilde x(y)\in[L_{m-1}^-+1,L_{m-1}^+]^d$
for each $y\in V_x$, so $R(y)$ and $R(\tilde x(y))$ are bounded from
above by $ R_{m-1}$ [recall $\tilde x(y) = x-y+(0,1,1,\ldots, 1)$].
Now applying Lemma~\ref{claim:main} with $\rho$ uniform on $V_x$, that
is, $\rho(y) =1/|V_x|$ for all $y\in V_x$, we have
\[
R(x) \leq9 \biggl(\frac{4N}{2^{m}} \biggr)^{d} R_{m-1} +
\frac
{18}{q} \biggl(\frac{4N}{ 2^{m}} \biggr)^{2d}
R_{m-1} \leq\frac
{27}{q} \biggl(\frac{4N}{ 2^{m}}
\biggr)^{d} R_{m-1}.
\]
We arrive at \eqref{eq:RN} by iterating the above inequality.
\end{pf}
In order to complete the proof of the upper bound in \eqref{hitmax},
fix $L\in(2^{n-1},2^n]$ with $n\le\theta_q/d$ and choose $N = n-3$.
In this case, $L \in[L_N^-,L_N^+]$, since
\begin{eqnarray*}
L_N^-&=& -2 +2^N \bigl(3 + \tfrac{ 1}{2} \bigr)
\leq2^{N+2} = 2^{n-1} < L \leq2^n \leq2 +
2^{N+3} \\
&\leq&2 +2^N \bigl(9 - \tfrac{ 1}{2}
\bigr)=L_N^{+}.
\end{eqnarray*}
Therefore, using \eqref{eq:TR} we have
$ T^{\s}(v^*;q) \leq R^\mathrm{ min}(v^*) \leq R_N$.
If we apply Lem\-ma~\ref{lem:res} with $m_0 =\lceil
\log_2[4(n-3)]\rceil$, we get
\[
R_N\le2^{n\theta_q -d {n\choose2} +O(\theta_q\log\theta_q)}R_{m_0}.
\]
The relaxation time of the East process on an interval $I$ of length
$O(m_0)$ is bounded by $2^{O(\theta_q m_0)}$ by \cite{CFM}, Theorem~2.
Also, by \cite{CFM}, Theorem~1, Proposition~3.2, this relaxation time is
of the same order as the mean time needed to put a vacancy in the
rightmost site of $I$ starting from the filled configuration. Now
following the derivation of (7.10) below we have the desired bound
\[
R_{m_0}\le2^{O(\theta_q m_0)}=2^{O(\theta_q\log\theta_q)}.
\]
%

\subsubsection{Proof of Lemma \texorpdfstring{\protect\ref{claim:main}}{7.1}}
The proof is based on an iterative procedure which generalizes our
construction in \cite{CFM}, Appendix A.2. Given $ y \in\L_x$, we
define $\widetilde\L_y:= [y_1+1,x_1] \times\prod_{i=2}^d [y_i,
x_i]$, $0_y\in\O_{\L_x}$ as the configuration with a single vacancy
located as $y$ and set
\begin{eqnarray*}
 B_y &:=& \{\eta\in\O_{\L_y} \dvtx \eta_y = 0
\} ,
\\
 B^x_y &:=& \{\eta\in\O_{\L_x} \dvtx
\eta_y = 0 \mbox{ and } \eta_z=1 \mbox{ for } z \in
\L_x \setminus\L_y\} ,
\\
 C_y^x&:=& \bigl\{\eta\in\O_{\L_x} \dvtx
\eta_y =0 \mbox{ and } \eta_z = 1 \mbox{ for } z \notin
\widetilde\L_y\cup\{y\} \bigr\}.
\end{eqnarray*}

%
\begin{figure}[b]

\includegraphics{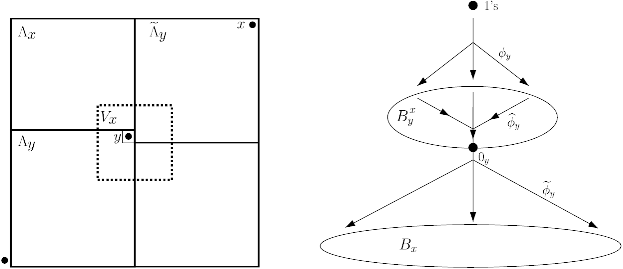}

\caption{Left: Geometry of the lattice $\Lambda_x$ with site
$x$ in the top right, and sub-lattices $\L_y$ and
$\widetilde\L_y$ for a $y\in V_x$. Right: Construction of
the unit flow $\theta_y=\phi_y+\hat\phi_y+\widetilde\phi
_y$.}\label{fig:animals}
\end{figure}

Let $\psi_y$ be the equilibrium unit flow in $\O_{\L_y}$ from
$\mathds{1}$ to $B_y$, whose energy equals $R(y)$. We now restrict to
$y \in V_x$ (thus implying in particular that the box
$\widetilde{\L}_y$ is not empty).
We introduce the flows $\phi_y, \widehat\phi_y,\widetilde\phi_y$
on $\O_{\L_x}$ (cf. Figure~\ref{fig:animals}), roughly described as
follows: $\phi_y$ is the unit flow from $\mathds{1}$ to $B_y^x$
obtained by mimicking $\psi_y$ on configurations which have no
vacancies outside $\L_y$, $\widehat\phi_y$ keeps the vacancy at $y$
fixed and reverses $\phi_y$ to clear all the other vacancies ($\phi
_y+\widehat\phi_y$ will become a unit flow from $\1$ to $0_y$), and
finally $\widetilde\phi_y$ is the unit flow from $0_y$ to $B_x$ which
mimics $\psi_{\tilde x(y)}$ by using only transitions inside
$\widetilde\L_y$. More precisely, we set
%
\begin{eqnarray}
 \phi_y(\s, \eta)&=& \cases{ \psi_y (
\s_{\L_y}, \eta_{\L_y}), &\quad $\mbox{if } \s_z,
\eta_z =1 \mbox{ for } z \in\L_x \setminus \L_y
,$ \vspace *{2pt}
\cr
0, &\quad $\mbox{otherwise} $, }\\
\widehat
\phi_y(\s,\h) &:=& \cases{ \phi_y \bigl(
\h^y,\s^y \bigr), &\quad  $ \mbox{if } \s,\h\in
B^x_y $,\vspace *{2pt}
\cr
0, &\quad $ \mbox{otherwise} $,}
\\
\widetilde\phi_y(\s,\h) &:=& \cases{ \psi_{\tilde x (y)}(
\widetilde\s,\widetilde\h), & \quad $\mbox{if } \s,\h \in C^x_y
$,\vspace*{2pt}
\cr
0, &\quad  $\mbox{otherwise} $,}
\end{eqnarray}
where $\widetilde\h\in\O_{\L_{\widetilde{x}(y)}}$ is defined as $
\widetilde\h_z :=
\h_{z+y-(0,1,1,\ldots,1)} $ for $z \in\L_{\tilde{x}(y)}$. Note that
$\widetilde\L_y-y+(0,1,1,\ldots, 1)= \L_{\tilde{x}(y)}$ and $\tilde
{x}(y) \in\L_x$.

\begin{claim}
\label{claim:zeroDiv} For each $y \in V_x$, the flow
$\theta_y := \phi_y +\widehat\phi_y + \widetilde\phi_y$ is a unit flow
from $\1$ to $ B_x$. In particular, $\Theta:= \sum_{y\in V_x}\rho(y)
\theta_y$ is a unit flow from $\1$ to $ B_x$.
\end{claim}
\begin{pf}
We prove that $\theta_y$
is a unit flow from $\1$ to $ B_x$, which trivially implies the thesis
for $\Theta$. Fix $y \in V_x$. Note that $y \neq (1,1,\ldots, 1)$ and
$y \neq x$ by our conditions on $V_x$.
Clearly, $\div\theta_y(\1) = 1$ by construction, it remains to show
that $\div\theta_y (\h) = 0$ for all $\h\notin B_x \cup\{\1\}$ and
$\div\theta_y(\h) \leq0$ for all $\h\in B_x$.
In general, we have $\div\theta_y = \div\phi_y + \div\widehat\phi
_y + \div\widetilde\phi_y$, while $ \div\phi_y(\h)$, $\div
\widehat
\phi_y(\h)$ and $\div\widetilde\phi_y(\h)$ equal, respectively,
\begin{eqnarray*}
&&\sum_{z \in\L_y\dvtx c_z^{\L_x , \mathrm{ min} }(\eta)=1}\phi_y \bigl( \eta, \eta
^z \bigr) ,\qquad \sum_{z \in\L_y\dvtx c_z^{\L_x , \mathrm{ min} }(\eta)=1}\widehat
\phi_y \bigl( \eta, \eta^z \bigr) , \\
&&\sum
_{z \in\widetilde\L_{y} \dvtx c_z^{\L_x , \mathrm{ min} }(\eta)=1} \widetilde\phi_y \bigl(\eta,
\eta^z \bigr).
\end{eqnarray*}

If $\h\in B_x$, then $\div\theta_y(\h) = \div\widetilde\phi_y(\h)$
and the latter equals $\div\psi_{\tilde x (y)}(\widetilde\h)$ if
$\h
\in C_y^x$ and zero otherwise. Since $\div\psi_{\tilde x
(y)}(\widetilde\h) \leq0$ for $\h\in C_y^x \cap B_x$ by definition of
the equilibrium flow, we conclude that $\div\theta_y(\h) \leq0$ for
all $\h\in B_x$.
We now distinguish several cases, always restricting to $\h\notin B_x
\cup\{\1\}$.

\begin{itemize}
\item \textit{Case $\h\notin B^x_y \cup C^x_y $}. By construction,
$\div\phi_y (\h) = \div\psi_y(\h_{\L_y})$ or $0$. Since $\psi
_y$ is a
unit flow from $\1$ to $B_y$, it is divergence free outside of $\1$ and
$B_y$, in particular $\div\phi_y (\h) = 0$. Also $\widehat\phi
_y(\h
,\cdot) \equiv0$ and $\widetilde\phi_y(\h,\cdot) \equiv0$. This
implies that $\div\theta_y =0$.

\item \textit{Case
$\eta\in C^x_y$ and $\eta\neq 0_y$}. We have $ \phi_y(\eta, \cdot
)\equiv0$ and $\widehat\phi_y(\eta, \cdot)\equiv0$. On the other
hand, $\div\widetilde\phi_y(\eta)= \div\psi_{\widetilde{x}(y)}
(\tilde\eta)=0$ since $\tilde\eta\notin B_{\widetilde{x}(y)}
\cup
\mathds{1}_{\L_{\widetilde{x}(y) }}$ (recall that $\eta\notin B_x$,
$\eta\neq0_y$).

\item \textit{Case
$\h\in B^x_y $ and $\eta\neq 0_y$}. Note that $B^x_y \cap C^x_y=0_y$,
so $\widetilde\phi_y(\eta, \cdot)\equiv0$.
Also $\phi_y(\s, \s')=0 $ if $\s,\s' \in B^x_y$ since $\psi_y$ is the
equilibrium unit flow in $\O_{\L_{y}}$ from $\1$ to $B_y$, otherwise
replacing $\psi_y$ by a flow which is identical on all edges except
between configurations in $B_y$, on which the new flow is identically
zero, would give rise to a unit flow from $\1$ to $B_y$ with lower
energy, contradicting the variational characterization of the
equilibrium unit flow. It follows that
\begin{eqnarray*}
\div\theta_y(\eta) &=& \mathop{\sum_{z \in\L_y \dvtx }}_{ c_z^{\L_x ,
\mathrm{ min} }(\eta)=1}
\bigl( \phi_y \bigl( \eta, \eta^z \bigr) + \widehat\phi
_y \bigl( \eta, \eta^z \bigr) \bigr)
\\
&=& - c_y^{\L_x, \mathrm{ min}} (\eta)\phi_y \bigl(
\eta^y, \eta \bigr) -\mathop {\sum_{z \in\L_y \dvtx \h^z\in B^x_y }}_{ c_z^{\L_x, \mathrm{ min} }(\eta)=1}
\phi _y \bigl( \eta^y, \bigl(\eta^z
\bigr)^y \bigr)
\\
&=& - \div\phi_y \bigl(\h^y \bigr) = 0,
\end{eqnarray*}
where in the second identity we have used $(\h^z)^y=(\h^y)^z$.
The last identity follows from the fact that $\h,\h^z\in B^x_y$ implies
$z\neq y$, and that $\eta^y_{\L_y} \notin B_y \cup\{\1_{\L_y} \}$,
hence $\div\psi_y( \eta^y_{\L_y} )=0$.

\item \textit{Case $\h=0_y $}. There are only $1+d$ transitions under
the East dynamics from state $0_y$: the unconstrained site $(1,1,\ldots
,1)$, as well as any of the $d$ upper--right neighbors of $y$, can update.
However, any transition with nonzero flow $\theta_y$ must change the
configuration only inside $\L_y\cup\widetilde\L_y$. Hence,
\begin{eqnarray*}
\div\theta_y (0_y) &=& \widehat\phi_y
\bigl(0_y,0_y^{(1,1,\ldots,1)} \bigr) + \widetilde
\phi_y \bigl(0_y,0_y^{y+(1,0,\ldots,0)} \bigr)
\\
&=& -\psi_y \bigl(\1,\1^{(1,1,\ldots,1)} \bigr) +
\psi_{\widetilde x(y)} \bigl(\1,\1 ^{(1,1,\ldots,1)} \bigr) = -1 + 1 = 0.
\end{eqnarray*}
\end{itemize}
\upqed\end{pf}
%
Given two flows $\theta, \theta' $ on $\O_{\L_x}$ we write $\theta
\perp\theta'$ if
$\theta\cdot\theta'\equiv0$, that is, $\theta$ and $\theta'$ have
disjoint supports.
Note that, given $y \neq z$ in $V_x$,
$B_y^x \cap B_z ^x= \varnothing$ and
$C_y^x \cap C_z ^x= \varnothing$. Hence, by definition of
$ \widehat\phi_y$, $ \widetilde\phi_y $ we get
%
\begin{equation}
\label{dampi} \widehat\phi_y \perp\widehat\phi_z ,\qquad
\widetilde\phi_y \perp \widetilde\phi_z \qquad \mbox{for any
$y \neq z$ in $V_x$}.
\end{equation}
To complete the proof of the lemma, we set
\[
\Phi: = \sum_{y\in V_x} \rho(y)\phi_y ,\qquad
\widehat\Phi:= \sum_{y\in
V_x} \rho(y)\widehat
\phi_y ,\qquad \widetilde\Phi: = \sum_{y\in V_x}
\rho (y)\widetilde\phi_y.
\]
Note that $\Theta= \Phi+ \widehat\Phi+ \widetilde\Phi$.
Due to Claim~\ref{claim:zeroDiv}, $\Theta$ is a unit flow in $\O_{\L
_x}$ from $\1$ to $B_x$. Moreover, by Thompson principle [cf. \eqref
{eq:thom}], Schwarz inequality and \eqref{dampi}, we get
%
\begin{eqnarray}
\label{torlo} R(x) & \leq&\cE(\Theta) \leq3 \cE(\Phi)+3 \cE(\widehat\Phi) + 3
\cE (\widetilde\Phi)
\nonumber
\\[-8pt]
\\[-8pt]
\nonumber
& \leq&3 \sum_{y\in V_x} \rho(y) \cE(\phi_y)
+ 3 \sum_{y\in V_x} \rho ^2(y)\cE(\widehat
\phi_y) + 3 \sum_{y\in V_x}
\rho^2(y)\cE (\widetilde\phi _y).
\end{eqnarray}
Let $\eta\in\O_{\L_x}$ with $\eta_{\L_x\setminus\L_y}=
\1_{\L_x\setminus\L_y}$ and let $z \in\L_y$.
Observe now that, 
$(\eta, \eta^z)$ is a possible transition for the East dynamics on
$\L
_x$ if and only if $(\eta_{\L_y}, \eta_{\L_y}^z)$ is a possible
transition for the East dynamics on $\L_y$, and in this case (since
$|\L
_x| \leq1/q$ and $1-q\leq e^{-q}$)
\[
r^{\L_x, \mathrm{ min} } \bigl(\eta,\eta^z \bigr)= (1-q)^{-|\L_x \setminus\L_y|}
r^{\L
_y, \mathrm{ min} } \bigl(\eta_{\L_y},\eta_{\L_y}^z
\bigr)\leq e r^{\L_y, \mathrm{
min} } \bigl(\eta_{\L_y},\eta_{\L_y}^z
\bigr).
\]
This implies that $\cE( \phi_y)\leq e \cE(\psi_y)= e R(y)$. Similarly,
by straightforward computations, one can prove that $\cE(\widehat\phi
_y) \leq(e/q) R(y)$ and
$\cE( \widetilde\phi_y) \leq(e/q)
R(\tilde x(y) ) $. Coming back to \eqref{torlo}, we get the thesis.


\subsection{Proof of \texorpdfstring{\protect\eqref{hitmin2}}{(2.13)}}
\subsubsection{Lower bound}
The lower bound follows by appealing to the combinatorial result of
Lemma~\ref{lemma:CDG} and
making a similar bottleneck argument as for the proof of the lower
bound in \eqref{eq:333}.
Fix $x\in{\mathbb Z}_+^d$ such that $\|x-x_*\|_1+1 \in[2^{n-1},2^n)$ where
$x_*=(1,1,\ldots,1)$.
Recall that $V_m$ is the set of configurations which can be reached,
starting from the configuration $\1$ on ${\mathbb Z}_+^d$, by East-like paths
in $Z_m$.
Let $V= \{\h_{\L} \dvtx \h\in V_{n-1}\}$ be the image of $V_{n-1}$ under
projection on the lattice $\L= \prod_{i=1}^d[1,x_i]$, then by Lemma~\ref{lemma:CDG} $\1_\L\in V $ and $\h\notin V$ for all $\h\in\O
_\L
$ such that $\h_{x} = 0$, since $x \geq2^{n-1}$ and $Y(n-1) = 2^{n-1}-1$.

It follows from the graphical construction (see Section~\ref{sec:graph}) that for any an event $\cA$ which belongs to the $\s
$-algebra generated by $\{\h_x(s)\}_{x \in\L}$ we have
$ \bbP^{{\mathbb Z}_+^d, \mathrm{ min}}_\h(\cA) = \bbP^{\L,\mathrm{ min}}_{\h
_{\L
}}(\cA)$.
In particular,\vspace*{1pt} we get
$ T^\mathrm{ min}(x;q) = \bbE_{\mathds1}^{\L,\mathrm{ min}}(\t_x) $ so that
(cf. the beginning of Section~\ref{s.oreste})
\[
\bbE_{\mathds1}^{\L,\mathrm{ min}}(\t_x) \geq\frac{c}{ \cD^\mathrm{
min}_{\L
}(\1_{V})}.
\]
Finally, observe that to escape the set $V$ a vacancy must be created
by a transition which is allowable under the East-like dynamics,
therefore, $\partial V \subseteq U := \{ \h\in\O_\L\dvtx |\h| = n-1\}$,
so using \eqref{rane}
\begin{eqnarray*}
\cD^\mathrm{ min}_{\L} (\1_V) &=& \sum
_{\h\in\partial V}\pi_\L(\h )\cK^\mathrm{ min} \bigl(
\h,V^c \bigr) \leq\sum_{\h\in U}
\pi_\L(\h) \mathop{\sum_{x\in\L\dvtx \h_x=1}}_{ c_x^{\L, \mathrm{ min}}(\h) = 1}q
\\
&\leq&\pi_\L(U) d n q \leq d n c_0(n-1,d)q^{n}
,
\end{eqnarray*}
where $c_0(n-1,d) = |U|$ is the number of configurations in $\O_\L$
with exactly $n-1$ vacancies.

\subsubsection{Upper bound}
The upper bound follows by Rayleigh's monotonicity principle combined
with Lemma~\ref{lem:capbnds}.
Fix $x\in{\mathbb Z}_+^d$ such that $\|x-x_*\|_1+1 \in[2^{n-1},2^n)$ where
$x_*=(1,1,\ldots,1)$, and let $\L= \prod_{i=1}^d[1,x_i]$. Lemma~\ref
{lem:capbnds} implies
\[
T^\mathrm{ min}(x;q)= \bbE_{\mathds1}^{\L,\mathrm{ min}}(
\t_x) \leq R^\mathrm{ min}_{\1, B_x},
\]
where $B_x = \{\h\in\O_{\L} \dvtx \h_x = 0\}$.
Rayleigh's monotonicity principle (see, e.g., \cite{Levin2008}, Theorem~9.12) implies that, for any set of conductances $\cC'(\h
,\xi
)$ defined on $\O_\L^2$ with $\cC'(\h,\xi) \leq\cC^\mathrm{ min}(\h
,\xi)$
for all $(\h,\xi) \in\O_\L^2$ the associated resistance satisfies
$R'_{\1, B_x} \geq R^\mathrm{ min}_{\1, B_x}$.
Consider the directed spanning tree as in Section~\ref{sec:tree}.
Let $\G$ be all the vertices in the branch from $r$ to $x$.
Now define new conductances by
\begin{eqnarray*}
\cC'(\h,\xi) = \cases{\cC^\mathrm{ min}(\h,\xi), &\quad $
\mbox{if } \h_{\L\setminus\G} = \xi _{\L
\setminus\G}=\1,$\vspace*{2pt}
\cr
0, &\quad $\mbox{otherwise.}$ }
\end{eqnarray*}
The resulting resistance graph is isomorphic to that of the East
process on $[1,|\G|]$.
So, if we let $T_\mathrm{ East}(|\G|;q)$ be the mean hitting time of $\h
_{|\G|} = 0$ in the one-dimensional process we have
%
\begin{equation}
\label{eq:pasqua2} T^\mathrm{ min}(x;q) \leq R^\mathrm{
min}_{\1, B_x} \leq R'_{\1, B_x} \leq c
T_\mathrm{ East}\bigl(|\G|;q\bigr) \leq2^{n\theta_q +O_n(1)},
\end{equation}
where the penultimate inequality is due to Lemma~\ref{lem:capbnds}
(with $d=1$) and the final inequality is due to previous bounds on the
mean hitting time in the East process (see, e.g., \cite{CFM}, Theorem~1 and
equations (2.6) and~(3.1)).

\begin{appendix}\label{app}
\section*{Appendix: On the rate of decay of the persistence function}
Consider the East-like process in ${\mathbb Z}^d$ and let $\t$ be the first
time that there is a legal ring
at the origin. Let $F(t):= \bbP_\pi(\t>t)$ be the \emph{persistence
function} (see, e.g., \cite{Harrowell,SE1}) and let
$A(t):=\var_\pi(e^{t\cL}\eta_0)^{1/2}$. Notice that, using reversibility,
\[
A(t/2)^2= \var_\pi \bigl(e^{({t}/{2})\cL}
\eta_0 \bigr)=\pi \bigl(\eta_0 e^{t \cL
}
\eta_0 \bigr)-p^2
\]
that is, it coincides
with the time autocorrelation at time $t$ of the spin at the origin.
In analogy with the stochastic Ising model \cite{Holley}, it is very
natural to conjecture that $A(t)$ and $F(t)$ vanish
exponentially fast as $t\to\infty$, with a rate \emph{equal} to the
spectral gap of
the generator $\cL$. Here, we
show that the rate of exponential decay of $F(t)$ and $A(t)$ coincide
in any dimension
and we prove the above conjecture in one dimension (i.e., for the East model).
%
\begin{theorema}
\label{persistence2}
Consider the East-like process on ${\mathbb Z}^d$. Then
\begin{eqnarray*}
\limsup_{t\to\infty} t^{-1} \log F(t)&=&\limsup
_{t\to\infty} t^{-1} \log A(t),
\\
\liminf_{t\to\infty} t^{-1} \log F(t)&=&\liminf
_{t\to\infty} t^{-1} \log A(t).
\end{eqnarray*}
In the one-dimensional case $d=1$,
\[
\lim_{t\to\infty} t^{-1} \log F(t)=\lim
_{t\to\infty} t^{-1} \log A(t)= -\gap(\cL).
\]
\end{theorema}
%
\begin{remarka}
As will be clear from the proof, the last statement applies also to the
constrained model in ${\mathbb Z}^d$, $d\ge1$, in which the constraint
at $x$
requires that all the neighbors of $x$ of the form $y = x - e$, $e \in
\cB$ contain a vacancy.
These models share with the one-dimensional East process the key feature
that, starting from a configuration with no vacancies in
$\L=[-L+1,0]^d$, at the time of the first legal ring at the origin,
all vertices in $\L$ have been updated at least once.
\end{remarka}
To prove the theorem, we first need two basic lemmas.
%
\begin{lemmaa}
\label{persistence}
For all $t>0$,
%
\setcounter{equation}{0}
\begin{equation}
\label{eq:A} \frac{1}{(p\vee q)^2}A^2(t/2) \le F(t) \le
\frac{1}{(p\wedge q)} A(t).
\end{equation}
In particular,
%
\begin{equation}
\label{eq:B} F(t)\le\frac{1}{(p\wedge q)} e^{-t/\trel({\mathbb Z}^d;q)}.
\end{equation}
\end{lemmaa}
%
\begin{remarka}
The above result considerably refines a previous bound given in \cite{CMRT},
Theorem~3.6.
\end{remarka}
\begin{pf*}{Proof of the Lemma~\ref{persistence}}
Clearly, \eqref{eq:A} implies \eqref{eq:B}. To prove \eqref{eq:A}, for
any $\eta\in\O$ we write
\[
\bbE_\eta \bigl(\eta_0(t)-p \bigr)= (
\eta_0-p)\bbP_\eta(\t> t) + \bbE_\eta \bigl(
\eta_0(t)-p\mid\t\le t \bigr) \bbP_\eta(\t\le t).
\]
By the very definition of the East-like process, the law of
$\eta_0(t)$ given that $\{\t\leq t\}$ is a $\operatorname{Bernoulli}(p)$. Hence, the
second term in the
RHS above is zero. Thus,
\begin{eqnarray*}
A(t)&=& \pi \bigl( \bigl[ \bbE_\eta \bigl(\eta_0(t)-p
\bigr) \bigr]^2 \bigr)^{1/2}
\\
&=& \pi \bigl((\eta_0-p)^2 \bbP_\eta(\t>
t)^2 \bigr)^{1/2}
\\
&\ge&(p \wedge q) \bbP_\pi(\t> t) = (p \wedge q) F(t),
\end{eqnarray*}
and the sought upper bound follows. Similarly,
\begin{eqnarray*}\qquad\qquad
A^2(t/2)&=& \pi \bigl((\eta_0-p)\bbE_\eta
\bigl(\eta_0(t)-p \bigr) \bigr)
=\pi \bigl((\eta_0-p)^2\bbP_\eta(\t>t)
\bigr)
\\
&\le& (p\vee q)^2 \bbP_\pi(\t> t)=(p\vee q)^2
F(t).\hspace*{98pt}\qed
\end{eqnarray*}
\noqed\end{pf*}
%
The second lemma specializes to the one-dimensional case and it extends
a coupling result proved in
\cite{CFM}, Section~1.2. Fix an integer $L$ and let $\L=[-L,0]$.
Consider the East process on
the negative semi-infinite lattice ${\mathbb Z}^-:=(-\infty,0]$, with initial
distribution $\mu_{\pi,\omega}$ given by the product of the equilibrium
measure $\pi$ on $\O_{(-\infty,-(L+1)]}$ and the Dirac mass on
$\omega\in\O
_\L$. Let also $\mu_{\pi,\omega}^t$
be the corresponding law at a later time $t>0$.
%
\begin{lemmaa}
\label{coupling}
Let $d_\L(t)=\max_\omega\|\mu_{\pi,\omega}^t -\pi\|_\mathrm{ TV}$, where
$\|\cdot\|_\mathrm{ TV}$ denotes the total variation distance. Then
%
\begin{equation}
\label{eq:persi1} d_\L(t)\le (1/p )^{L+1}F(t).
\end{equation}
\end{lemmaa}
\begin{pf}
Let $\mathds1$ be the configuration in $\O_\L$ identically equal to
one and let $F_{\pi,\mathds1}(t)=\int \,d\mu_{\pi,\mathds
1}(\eta)\bbP_\eta(\t>t)$. Let $\eta^{\s,\omega}(\cdot)$ be
the East process on ${\mathbb Z}^-$ given by the graphical
construction, started from the
initial configuration equal to $\s$ on $(-\infty, -(L+1)]$ and to
$\omega$ on $\L$. Let also $X^{\s,\mathds1}_t$ be the largest $x\in
\L$ such
that, starting from the configuration equal to $\s$ on $(-\infty,
-(L+1)]$ and to
$\mathds1$ on $\L$, there has been a legal ring at $x$ before time
$t$. If no point in $\L$ had a legal ring before $t$, we set
$X^{\s,\mathds1}_t=-(L+1)$.
%
\begin{claima}
For all $\s,\omega,\omega'$ and all $t$, the two configurations
$\eta^{\s,\omega}(t), \eta^{\s,\omega'}(t)$
coincide on the semi-infinite interval $(-\infty,X^\s_t]$.
\end{claima}
If we assume the claim, we get that
\begin{eqnarray*}
\max_\omega\bigl\|\mu_{\pi,\omega}^t -\pi
\bigr\|_\mathrm{ TV}&\le& \max_{\omega,\omega'}\bigl\|\mu_{\pi,\omega}^t-
\mu_{\pi,\omega'}^t\bigr\| _\mathrm{ TV}
\\
&\le&\max_{\omega,\omega'}\int \,d\pi(\s)\bbP \bigl(\eta^{\s,\omega}(t)
\neq\eta^{\s,\omega'}(t) \bigr)
\\
&=&\int \,d\pi(\s)\bbP \bigl(X^{\s,\mathds1}_t<0
\bigr)=F_{\pi,{\mathds1}}(t) \le (1/p)^{L+1}F(t).
\end{eqnarray*}
The claim is proved inductively. By the oriented character of the East
process, the two configurations $\eta^{\s,\omega}(t),\eta^{\s
,\omega'}(t)$
will remain equal inside the semi-infinite interval $(-\infty,-(L+1)]$
for any $t\ge0$. It is also clear by the graphical construction that
once the vertex $x=-L$ is updated [at the same time for both
$\eta^{\s,\omega}(\cdot),\eta^{\s,\omega'}(\cdot)$], the two
configurations become
equal in $(-\infty,-L]$ and stay equal there forever. By repeating
this argument for the vertices $-L+1,-L+2,\ldots,$ we get the claim.
\end{pf}
%
%
\begin{pf*}{Proof of Theorem~\ref{persistence2}}
The first part
follows at once from Lemma~\ref{persistence}. To prove the second part,
we observe that, using again Lemma~\ref{persistence}, it is enough to
show that, for the East model,
\[
\liminf_{t\to\infty}t^{-1}\log F(t)\ge-\gap(\cL).
\]
For this purpose, fix an integer $L$, let $\L=[-L,0]$ and let $\phi$
denotes the
eigenvector of $\cL^\mathrm{ max}_\L$ with eigenvalue
$-\gap(\cL^\mathrm{ max}_\L)$, normalized in such a way that $\var_\pi
(\phi
)=1$. We start by observing that
%
\begin{equation}
\label{eq:persi2} \var_\pi \bigl(e^{t\cL}\phi \bigr)\ge
e^{-2t \cD(\phi)}\ge e^{-2t
\cD_\L^\mathrm{ max}(\phi)} = e^{-2t
\gap(\cL_{\L}^\mathrm{ max})}.
\end{equation}
To prove the first bound, we use the spectral theorem for the
self-adjoint operator
$\cL$. Let $\nu_\phi(\cdot)$ be the spectral measure (for the infinite
system) associated to $\phi$. Clearly,
$\nu_\phi$ is a probability measure. Using Jensen's inequality, we get
\[
\var_\pi \bigl(e^{t\cL}\phi \bigr)=\int
^{\infty}_0 e^{- 2t
\lambda}\,d
\nu_\phi( \lambda)\ge e^{-2t \int^{\infty}_0\lambda \,d\nu
_\phi(\lambda)}= e^{-2t \cD
(\phi)}.
\]
We now prove an upper bound on $\var_\pi (e^{t\cL}\phi )$ in
terms of the persistence function $F(t)$.

Recall
the definition of the law $\mu_{\pi,\omega}^t$ in Lemma~\ref{coupling}. Using reversibility and the fact that $\pi_{\L}(\phi
)=0$, we get
%
\begin{equation}\qquad
\label{eq:persi3} \var_\pi \bigl(e^{t\cL}\phi \bigr)=
\mathrm{ Cov}_\pi \bigl(\phi, e^{2t\cL}\phi \bigr) =\sum
_{\omega\in\O_\L}\pi(\omega)\phi(\omega) \bigl[
\mu_{\pi
,\omega}^{2t}( \phi)-\pi_\L (\phi) \bigr].
\end{equation}
Above we used the oriented character of the East model to get that the
marginal on $\O_{\L}$ of the law at time $t$ of the East process on
${\mathbb Z}$ coincides with
the same marginal for the process on the half lattice ${\mathbb Z}^-$.
Using Lemma~\ref{coupling}, the RHS of \eqref{eq:persi3}
can be bounded from above by
%
\begin{equation}
\label{eq:C} \sum_{\omega\in\O_\L}\pi(\omega)\phi(\omega)
\bigl[\mu_{\pi,\omega}^{2t}(\phi)-\pi_\L(\phi) \bigr]
\le \frac{1}2 \|\phi\|_\infty^2 (1/p)^{L+1}F(2t).
\end{equation}
In conclusion, by combining \eqref{eq:persi2}, \eqref{eq:persi3} and
\eqref{eq:C} we
get that
\[
F(2t)\ge\frac{2}{\|\phi\|_\infty^2}p^{L+1}e^{-2t\gap(\cL_\L^\mathrm{ max})},
\]
which, in turn, implies that
\[
\liminf_{t\to\infty}t^{-1}\log F(t)\ge-\gap \bigl(
\cL_\L^\mathrm{ max} \bigr)\qquad \forall L\ge1.
\]
Since $\gap(\cL_\L^\mathrm{ max})\to\gap(\cL)$ as $L\to\infty$ (see
\cite{CMRT}), we get that
\[
\liminf_{t\to\infty}t^{-1}\log F(t) \geq-\gap(\cL),
\]
as required.
\end{pf*}
\end{appendix}

\section*{Acknowledgments}
We would like to thank J. P. Garrahan for some inspirational discussion
about kinetically constrained models and
for his encouragement to analyze
the generalizations of the East process to higher dimensions.
We would also like to thank C. Toninelli for pointing out an
unjustified remark in the \hyperref[app]{Appendix} of the preprint.
Finally, the
authors would like to thank the anonymous referee whose comments and corrections
have allowed us to improve the presentation.

%

\printaddresses
\end{document}